\documentclass{article}
\usepackage{graphicx} 

\title{CMZ_10_23}
\author{elvira.zappale }
\date{october 2023}

\usepackage{amsmath, amssymb ,amsthm, amsfonts, amsgen,
mathrsfs}
\usepackage{mathtools}
\usepackage{color}
\usepackage{esint,amssymb}
\usepackage[notref, notcite]{showkeys}
\usepackage{color}
\usepackage{tikz}
\usepackage{tikz-3dplot}
\usetikzlibrary{3d,calc,shapes}
\usepackage{bbm}
\usepackage{cite}
\definecolor{vg}{rgb}{0.0, 0.26, 0.15}
\numberwithin{equation}{section} 
\def\weakstar{\buildrel\star\over\rightharpoonup}
\def\e{{\varepsilon}}
\def\O{{\Omega}}
\def\o{{\omega}}
\def\x{{\times}}

\def\weak{\rightharpoonup}

\def\F{{\sl F}}
\def\G{{\Gamma}}

\def\U{{\mathcal U}}

\def\R{{\cal R}}
\def\RI{{\mathbb R}}

\newtheorem{Theorem}{Theorem}[section]
\newtheorem{Lemma}[Theorem]{Lemma}
\newtheorem{Proposition}[Theorem]{Proposition}

\newtheorem{Remark}[Theorem]{Remark}
\newtheorem{Definition}[Theorem]{Definition}
\newcommand{{\rr}}{{\mathbb R}}

\newcommand{\N}{{\mathbb N}}
\newenvironment{@abssec}[1]{%
     \if@twocolumn
       \section*{#1}%
     \else
       \vspace{.05in}\footnotesize
       \parindent .2in
         {\upshape\bfseries #1. }\ignorespaces
     \fi}
     {\if@twocolumn\else\par\vspace{.1in}\fi}
\newcommand\keywordsname{Key words}
\newcommand\AMSname{AMS subject classifications}

\newcommand{\domb}{(\O^b;\RI^3)}

\newcommand{\hin}{\hbox{ in }}

\newcommand{\RR}{\mathbb R}

\begin{document}

\title{ Asymptotic analysis of a clamped thin multidomain allowing for fractures and discontinuities}
\author{Gabriel Carvalho\footnote{Instituto Superior T\'{e}cnico, University of Lisbon,  Av. Rovisco Pais $1049-001$ Lisboa, e-mail: gabriel.v.carvalho@tecnico.ulisboa.pt },  $\,$ Jos\'e Matias \footnote{Department of Mathematics, Instituto Superior T\'{e}cnico, University of Lisbon, Av. Rovisco Pais $1049-001$ Lisboa, email: jose.c.matias@tecnico.ulisboa.pt}, $\,$ and Elvira
Zappale\footnote{Department of Basic and Applied Sciences for Engineering, Sapienza -  University of Rome, via
A. Scarpa,  16 (00161) Roma (SA), e-mail:
elvira.zappale@uniroma1.it}}\date{ } \maketitle
\begin{abstract}   We consider a thin  multidomain of $\RI^3,$ consisting of
a vertical rod upon a horizontal disk. The equilibrium configurations of the thin hyperelastic 
multidomain, allowing for fracture and damage, are described by means of a  bulk energy density of the kind
$W(\nabla U)$, where $W$ is a Borel function with linear growth 
and $\nabla U$ denotes the gradient of  the displacement, i.e. a
vector valued function $U:\Omega \to \mathbb R^3$. By assuming that the two volumes tend to zero,  under suitable
boundary conditions and loads, and suitable assumptions of the rate of convergence of the two volumes, we prove that the limit model is well posed
in the union of the limit domains, with
 dimensions, respectively, $1$ and $2$.

\par
\noindent2020 \AMSname: 49J45, 74B20, 74K10, 74K20, 74K30, 74K35, 78M30,
78M35.
\end{abstract}
\section{Introduction}

Thin structures arise in many applications, such as crystal plasticity, modeling of magneto-elastic devices, ribbons, heterogeneous multi-structures, etc. and their study is usually 
addressed through dimensional reduction techniques, where an 
asymptotic analysis from high dimensions to lower ones is conducted. In this process, one looks to
find precise relations between the initial models and the limit
ones (see among a much wider literature \cite{ABP}, \cite{AT}, \cite{BZZ}, \cite{BZ}, \cite{BZ2}, \cite{BFF}, \cite{BDE}, \cite{BDV}, \cite{Cal}, \cite{CMMO},\cite{C2}, \cite{DV},\cite{EKK}, \cite{EPZ}, \cite{FORS}, \cite{FRP}, \cite{FRJM},
\cite{LeDR1}, \cite{ M, MS00, MS0, MS}, 
  and the references therein).
For a recent survey with an analytical perspective we refer to \cite{BDFFFLM}, while we refer to \cite{C1, C2} for the basis of the theory.
Following the dimension reduction models for the junction of two thin cylinders that have been addressed previously in \cite{FZ}, 
\cite{GGLM1},
\cite{GMMMS1},\cite{GMMMS2},
\cite{GPP},\cite{G1}, \cite{HL},\cite{K}, \cite{KMM} 
and \cite{LeD},
we consider a thin multidomain of
$\RI^3$
 with a  bulk energy density of the
kind $W(\nabla U)$, where $W$ is a Borel function  with linear growth, and $\nabla U$ denotes the gradient of a field
function $U \in W^{1,1}$. By assuming that the volumes of the two cylinders
tend to zero with different assumptions on their rates, under suitable boundary conditions on
the top of the vertical cylinder and on the lateral surface of the
horizontal cylinder, and under suitble body loads, we derive the limit energy in the space $BV$, allowing for discontinuities, damage and fractures, thus generalizing to the case of multistructures,  the results contained in \cite{BZZ}, leaving for a forthcoming paper the bending effects. We emphasize that fracture may appear in the limit model. Moreover, in contrast with \cite{BF, CMMO}, in this work we did not consider any initial surface energy.
More precisely, being $r_n$ the radius of the top cylinder and $h_n$ the height of the bottom cylinder and $\ell = \lim_{n \to \infty}\frac{h_n}{r_n^2}$, we 
analyse the cases $\ell \in ]0, +\infty[, \ell = 0$ and $\ell = +\infty$ (cf. \eqref{loadq}, \eqref{loadinfty} and \eqref{load0}). In this last case,  due to our techniques \color{black}, we consider only some specific setting, i.e. we deal with two different situations, namely the case where $\lim_{n\to \infty}\frac{h_n}{r_n} = 0$ and $\lim_{n\to \infty}\frac{h_n}{r_n} = \infty$ (we refer to Section \ref{MR} to more comments on this point).

Clearly the different values of $\ell$ lead to very different limit models, and despite the fact the initial thin $3D$ multidomain is a whole with a vanishing common interface between the top and bottom cylinder, in the limit process in most of the regimes described by $\ell$ one might end with two decoupled domains, namely neither a condition appear in deformations at the common point between the $1D$ domain above and the $2D$ domain below in terms of the limit function spaces nor an energy penalization at this point (which is the expected modeling for fields with free discontinuities), see Theorem \ref{generalrepresentation}. \color{black}
In any situation described above by $\ell$,  we prove that
the limit problem is well posed in the union of the limit domains,
with
 dimensions  $1$ and $2$, respectively. Applications of this model problem, in the case $\ell \in ]0,+\infty[$, can be found in \cite{GZ}.
 \\
 Weakly star converging sequences in the Sobolev space $W^{1,1}$ may develop concentrations and, when dealing with linear growth integrands, the effect is captured in the limit result in BV, through the recession function of the integrand. There are different notions of recession function but, when addressing lower semicontinuity and relaxation problems in $BV$ for integrands with linear growth  (cf. \cite{KR}), it is adequate to  define its recession function as a $limsup$ (see definition \eqref{S101rec}).

Our work relies on $\Gamma$-convergence techniques ( see \cite{DM} for a compreensive treatment) and on \cite{BFMTraces}, where the problem of relaxation of functionals in BV under trace constraints was first addressed in a wide generality. In particular, since our main focus consists in the modeling of the thin multidomain, we impose assumptions on our energy density completely analogous to the ones in \cite{BFMTraces}, but less restrictions could be imposed on the recession functions (see Remark \ref{refremweak}). \color{black} 
For the sake of completeness, we finally address the super-linear case (i.e. deformations $U \in W^{1,p},\; p >1$, $W$ with growth of order p ) for the cases $\ell = 0$ and $\ell = +\infty$, which were not addressed previously (cf. \cite{GZNODEA}).
The overall plan of this work is as follows.
In the following subsection, after having introduced the problem in a
thin multidomain, we reformulate it on a fixed domain through
appropriate rescalings of the kind proposed by P.G. Ciarlet and P.
Destuynder in \cite{CD} and we state our main result, Theorem \ref{generalrepresentation}.  In section 2 we fix notation and state and prove some preliminary results, while section 3 is devoted to the proof of our main theorem. Finally, in section 4 we briefly present some results in the super-linear case. 


\subsection{The original problem and the rescalings}\label{The re-scalings}
In the sequel,  $\mathbb{R}^{h
\x k}$, denotes  the set of $h\x
k$-matrices. In our setting, $h=3$, $k=1,2,3$ and $x=(x_1,x_2,x_3)
=(x_\alpha,x_3)$ denotes a generic point of $\mathbb
R^3$; the gradient with
respect to the first $2$ variables is denoted by $\nabla_{\alpha}$
while the first
derivative with respect to the last variable  is represented
 by $\nabla_3$. 

We write a generic element  in \(\mathbb{R}^{3\x 3}\)   as  \(M=(m_{i,j})_{1\leq i,j\leq 3}\). For our purposes, it is convenient to decompose a given
matrix  \(M=(m_{i,j})_{1\leq i,j\leq 3}\)
as follows: we set   
$M_\alpha :=(m_{i,j})_{1\leq i \leq 3,1\leq j\leq 2}
\in \mathbb{R}^{3\x
2}$, $M_3 =(m_{i,3})_{1\leq i\leq 3}\in \mathbb R ^3$.
\begin{equation}\label{matrixdef}
\begin{aligned}
M=\left(
M_\alpha \;\;| \;\; M_3
\right),
\end{aligned}
\end{equation}
In particular if $Id$ denotes the identity in $\mathbb R^{3\times 3}$, by $\mathbb R^{3\times 2}\ni Id_\alpha$ we denote its first two columns.
Further, \(|M|\) denotes the norm of \(M\) given by $ \big(\sum_{i,j=1}^3 |m_{i,j}|^2\big)^{1/2}.$

Let $\omega \subset\RI^2$ be a  bounded, open, and connected
set, with Lipschitz boundary,
such that the origin in $\RI^2$, denoted by $0_\alpha$, belongs
to
$\omega$.  Furthermore for the sake of exposition we assume that $\mathcal L^2(\omega)=1$. Let $\left(r_n\right)_{n \in \mathbb N}$,
$\left(h_n\right)_{n \in \mathbb N}\subset]0,1[$ be two
sequences such that
\begin{equation}\label{hrzero}
\lim_{n }h_n=0= \lim_{n }r_n,
\end{equation}
and 
\begin{equation}\label{ell}
\lim_{n} \frac{h_n}{r_n^2}=\ell \in [0,+\infty].
\end{equation}
For every  $n \in \mathbb N$, consider
a thin multidomain consisting of a union between two vertical
cylinders, one
placed upon the other, as follows:  $\Omega_n=\Omega_n^a\cup\Omega_n^b$
($a$ stands for
``above" and $b$ for ``below"), where  $\Omega_n^a=r_n\omega\times
[0,1[$ with
small cross section $r_n\omega$ and constant height,
$\Omega_n^b=\omega\times ]-h_n,0[$ with small thickness $h_n$
and
constant cross section (see Figure \ref{fig:domain1}).

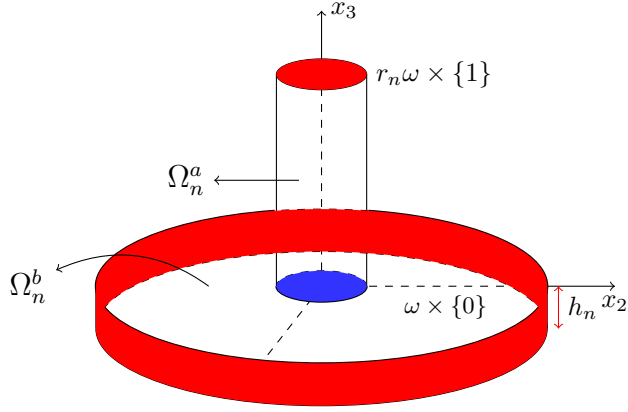
\begin{figure}[ht]
\centering
\tdplotsetmaincoords{70}{90}
    \begin{tikzpicture}[tdplot_main_coords,scale=3]
    
        \tikzstyle{junction}=[fill=blue!80,fill opacity = 0.5]
        \tikzstyle{dirichlet}=[fill=red,fill opacity=0.5]
        \tikzstyle{side}=[fill=green,fill opacity=0.5]
        \def\hn{0.2}
        \def\rn{0.2}
        \coordinate (O) at (0,0,0);
        \coordinate (H) at (0,0,1);
        \coordinate (Bn) at (0,0,-\hn);
        \coordinate (Hn) at (0,0,\hn);
        \coordinate (A) at (0,-1,0);
        \coordinate (C) at (0,1,0);

        \node at (O) {$\bullet$};
        \draw[dashed] (O) -- (0,0,1);
        \draw[->] (0,0,1) -- (0,0,1.3) node[anchor=west]{$x_3$};
        \draw[dashed] (O) -- (0,1,0);
        \draw[->] (0,1,0) -- (0,1.3,0) node[anchor=north ]{$x_2$};
        \tdplotsetrotatedcoords{-15}{0}{0}
        \draw[dashed,tdplot_rotated_coords] (O) -- (1,0,0);
        \draw[->,tdplot_rotated_coords] (1,0,0) -- (1.3,0,0) node[anchor=east]{$x_1$};

        \tdplotdrawarc[thick]{(O)}{\rn}{-90}{90}{}{};
        \tdplotdrawarc[tdplot_main_coords,dashed]{(O)}{\rn}{90}{270}{}{};
        \draw (H) circle (\rn);
        \draw ($\rn*(A)$) -- ($\rn*(A)+(H)$);
        \draw ($\rn*(C)$) -- ($\rn*(C)+(H)$);

        \tdplotdrawarc{(Bn)}{1}{-90}{90}{}{};
        \tdplotdrawarc[tdplot_main_coords,dashed]{(Bn)}{1}{90}{270}{}{};
        
        \tdplotdrawarc[thick]{(O)}{1}{-168}{168}{}{};
        \tdplotdrawarc[tdplot_main_coords,dashed]{(O)}{1}{168}{192}{}{};

        \draw (A) -- ($(A)-(Hn)$);
        \draw (C) -- ($(C)-(Hn)$);

        \node at (0,0.55,-0.1) {\small$\omega\times\{0\}$};
        \fill[junction] (O) circle (\rn);

        \draw[<->,color=red] ($(C) + (0,0.05,0)$) -- ($(C)-(Hn)+ (0,0.05,0)$) node[anchor=south west,black]{$h_n$};
        \node at (H) {$\bullet$};
        \fill[dirichlet] (H) circle (\rn) node[right=6mm,black,text opacity=1]{$\tiny r_n\omega\times\{1\}$};

        \fill[dirichlet] (0,1,0) -- (0,1,-\hn) arc (90:270:1) -- (0,-1,0) arc (270:90:1);
        \fill[dirichlet] (0,1,0) -- (0,1,-\hn) arc (450:270:1) -- (0,-1,0) arc (270:450:1);

        \node (Ob) at (0,-1.3,0) {\large$\Omega^b_n$};
        \draw[] (0,-0.5,0) edge[bend right,->] node {} (Ob);
        
        \node (Oa) at (0,-0.6,0.5) {\large$\Omega^a_n$};
        \draw[] (0,-0.1,0.5) edge[right,->] node {} (Oa);

    \end{tikzpicture}

\caption{Multidomain $\Omega_n$}
\label{fig:domain1}
\end{figure}

%
%

We  consider a Borel energy density \(W :  \mathbb{R}^{3\x 3}\to 
\mathbb{R}\)
 satisfying the following
assumptions:
\begin{align}
& \frac{1}{C} |M| -C  \leq W(M) \leq C(1+|M|),\quad\forall M\in \mathbb{R}^{3\x
3}; \label{coerci}
\end{align}
\noindent for some \(C>0\). 


%

\noindent 

For every $n\in \mathbb N$, the energy
functionals we are interested in are of the form \(I_n:W^{1,1}(\Omega_n;\mathbb R^3)\to \mathbb R\) defined by
\begin{equation}\label{principal energy}
\begin{aligned}I_n[U_n]:=  &\int_{\Omega_n} W\left(\nabla U_n(x)\right) dx + \int_{\Omega_n}H_n^{{\ell}} (x)\cdot U_n(x) dx= \\
&\int_{\Omega_n}
\left( W\left(\nabla_{\alpha}U_n(x) \;\; | \;\; \nabla_3 U_n(x)\right)+ H_n^{{\ell}} (x)\cdot U_n(x)\right) dx,
\end{aligned}
\end{equation}
where $H_n^{\ell} \in L^\infty(\Omega_n;\mathbb R^3)$ and the last summand represent the loads, that will be specified in the sequel, according to the considered scaling.

In addition, we require the admissible deformation  $U_n\in W^{1,1}(\Omega_n;\mathbb R^d)$ to satisfy the
Dirichlet boundary condition $c^{a, \ell}+ d^{a, \ell}x_\alpha$ on the top of
$\Omega_n^a$, and $f^{b, \ell}+  g^{b, \ell}x_3$ on the lateral surface of
$\Omega_n^b$, for some $c^{a, \ell}\in \mathbb{R}^3$, $d^{a, \ell}\in
\mathbb{R}^{3\times 2}$ 
and $f^{b, \ell}, g^{b, \ell}\in W^{1,1}(\omega;\mathbb R^3)$.
 In Remarks \ref{general_dirichlet1}, \ref{general_dirichlet2} and \ref{general_dirichlet3} below we make more detailed comments about the conditions we impose on the deformations at the boundary.
\color{black}

More general loads can be imposed, but for the sake of exposition we consider the above simplified case. We refer to Remark \ref{remloads} below for a more detailed discussion. \color{black}

Our main goal is to study the asymptotic behavior as \(n\to\infty\) of the minimization problem
\begin{equation}\label{eq:original}
\inf_{\begin{array}{cc}U_n\in W^{1,1}(\Omega_n;\mathbb
R^3) \\ {\hbox{under 
the above boundary conditions}}
\end{array}} I_n[U_n].
\end{equation}

As usual in dimension-reduction problems, we reformulate this
minimization problem on a fixed
domain through an appropriate rescaling  which maps $\Omega_n$ into $\Omega :=\omega \times
]-1,1[$. 

 Precisely,  we set
\begin{equation*}
 u_n(x)=\left\{
\begin{array}{ll} u^{a}_n(x_\alpha,x_3)=
 U_n(r_nx_\alpha,x_3),\quad (x_\alpha,x_3)\in\Omega^a :=\omega\times [0,1[, \\\\
 u^{b}_n(x_\alpha,x_3)= U_n(x_\alpha,h_n x_3), \quad
(x_\alpha,x_3)\in\Omega^b :=\omega\times ]-1,0[.
\end{array}\right.
\end{equation*}

\noindent Then, $ u^a_n\in
W^{1,1}(\Omega^a;\mathbb R^3)$ assumes the rescaled Dirichlet boundary condition
$c^{a,\ell} +r_nd^{a,\ell} x_\alpha$ 
on the top of $\Omega^a$, while $ u^b_n\in
W^{1,1}(\Omega^b;\mathbb R^3)$ assumes the rescaled Dirichlet boundary condition
$f^{b,\ell} +h_ng^{b,\ell}x_3$ 
on the lateral boundary of $\Omega^b$.
 Moreover,
 $ u_n=( u^a_n, u^b_n)$  satisfies
the 
junction condition 

\begin{equation}\label{eq:bc}
 u^a_n(x_\alpha,0)= u^b_n(r_nx_\alpha,0), \hbox{ for a.e.
}x_\alpha\in\o,
\end{equation}
In other words, \(u_n \in\mathcal{U}_n^\ell \), where 
\begin{align}\label{Vn}
\mathcal{U}^{\ell}_n=\Big\{ &(u^a, u^b) \in \left(c^{a,\ell}+r_nd^{a,\ell}x_\alpha 
+W_a^{1,1}
(\Omega^a;\RR^3)\right)\x \left(f^{b,\ell}+h_n g^{b,\ell}x_3
+W^{1,1}_b(\O^b;\RR^3)\right):
\\ &\qquad 
u^a \text{ and } u^b \text{ satisfy \eqref{eq:bc}}\Big\} \nonumber
\end{align}
with  $W_a^{1,1}(\Omega^a;\mathbb
R^3)$  the
closure with respect to the strong topology of $W^ {1,1}(\Omega^a;\mathbb R^3)$ of 
$$\left\{u^a\in C^\infty(\overline{\Omega^a};\RR^3)\,:\,
 u^a=\mathbf{0} \hbox{ in a  neighbourhood of }\omega\x\{1\}\right\}$$
 and respectively
  $W_b^{1,1}(\Omega^b;\RR^3)$  is the closure with respect to the strong topology of $W^{1,1}(\Omega^b;\mathbb R^3)$ of 
  $$\big\{u^b\in C^\infty(\overline{\Omega^b};\RR^3)\,:\,
 u^b={\bf 0} \hbox{ in a neighbourhood of } \partial\omega\times]-1,0[\big\}.$$
 
 Without loss of generality, one can assume that
 \begin{equation}\label{f=0=g}f^{b,\ell}={\bf 0}\hbox{ a.e. in }B,
 \end{equation}
for some $2$-dimensional ball $B$ such that $0_\alpha\in B\subset
\subset \omega$ ($\bf 0$ and $0_\alpha$ denote the null vectors in $\mathbb R^3$ and $\mathbb R^2$, respectively).

We consider now three functionals, depending on the superscript $\ell$ in \eqref{ell}, as rescalings of the energy $I_n$ in \eqref{principal energy}, and study their asymptotic behaviour.

\begin{Definition}
    \label{def:energies}
    Let $\ell \in [0,+\infty]$ be as in \eqref{ell}, i.e. the limit ratio between the height of the cylinder $\Omega^b$ and the area of the cross section of the cylinder $\Omega^a$.
     
    For $\ell\in\{q,\infty,0\}$ with $q\in]0,+\infty[$, let $F_n^\ell: \mathcal U_n^\ell \to \mathbb R$ be defined as follows:
\begin{alignat}{2}
    &F_n^q:\mathcal U_n^q \to \mathbb R, \quad &&F_n^q( u^a_n, u^b_n) := K_n^{a,q}( u^a_n)+ \frac{h_n}{r_n^{2}}K_n^{b,q}(u^b_n), 
    \label{eq:Fnq}
    \\
    &F_n^\infty:\mathcal U_n^\infty \to \mathbb R, \quad 
    &&F_n^\infty( u^a_n, u^b_n) := K_n^{a,\infty}( u^a_n)+ \frac{h_n}{r_n^{2}}K_n^{b,\infty}(u^b_n), 
    \label{eq:Fninfty}
    \\
    &F_n^0:\mathcal U_n^0 \to \mathbb R, \quad &&F_n^0( u^a_n, u^b_n) := \frac{r_n^{2}}{h_n}K_n^{a,0}( u^a_n)+ K_n^{b,0}(u^b_n), 
    \label{eq:Fn0}
\end{alignat}
where $\mathcal U_n^q$, $\mathcal U_n^\infty $ and $\mathcal U_n^0 $ are as in \eqref{Vn} with

\begin{alignat}{4}
    \label{dirichletq}
    &c^{a, q}\in \mathbb{R}^3, \quad &&d^{a,q}=({\bf 0}, {\bf 0}) \in \mathbb R^{3\times 2}, \quad &&f^{b, q}\in W^{1,1}(\omega;\mathbb R^3), \quad &&g^{b, q}={\bf 0}\in \mathbb R^3,
    \\
    \label{dirichletinfty}
    &c^{a, \infty}\in \mathbb{R}^3, \quad &&d^{a,\infty}=({\bf 0}, {\bf 0}) \in \mathbb R^{3\times 2}, \quad &&f^{b, \infty}(x_\alpha)=(x_\alpha,0), \quad &&g^{b, \infty}=(0_\alpha,1),
    \\
    \label{dirichlet0}
    &c^{a, 0}=(0_\alpha,1), \quad &&d^{a,0}= Id_\alpha,
    \quad &&f^{b, 0}\in W^{1,1}(\omega;\mathbb R^3), \quad &&g^{b, 0}={\bf 0} \in \mathbb R^{3},
\end{alignat}
and where
\(K_n^{a,\ell}:  W^{1,1}(\Omega^a;\RR^3)\to \mathbb R\) and 
\(K_n^{b,\ell}:  W^{1,1}(\Omega^b;\RR^3)\to \mathbb R\) are the
functionals defined, respectively, by
\begin{equation}\label{kna}
K_n^{a,\ell}[ u^a] := \int_{{\Omega}^a}
\left( W\left(\displaystyle{\frac{1}{r_n}\nabla_{\alpha}u^{a}(x)} \;\; | \;\;
\displaystyle{\nabla_{3}u^{a}}(x)\right)+ H_n^{a,\ell}(x) \cdot u^a(x)\right)dx,
\end{equation}\medskip

\begin{equation}\label{knb}
K_n^{b,\ell}[ u^b]:= \int_{{\Omega}^b}
\left(W\left(\displaystyle{\nabla_{\alpha}u^b(x)} \;\; | \; \;
\displaystyle{\frac{1}{h_n}\nabla_{3}u^b(x)}
\right)+ H_n^{b,\ell}(x)\cdot u^b(x)\right)dx,
\end{equation}
with $H_n^{a,\ell}(x):= H_n^\ell(r_n x_\alpha, x_3)$ if $(x_\alpha, x_3) \in \Omega^a$,
 and $H_n^{b,\ell}(x):= H_n^\ell(x_\alpha, h_n x_3)$ if $(x_\alpha,x_3) \in \Omega^b$.

\end{Definition}

\noindent 
Finally, the rescaled minimization problems corresponding to \eqref{eq:original}
read as \begin{equation}\label{minimumproblem_n}
\inf_{(u^a_n, u^b_n)\in \mathcal{U}_n^\ell}\left\{F_n^\ell(u_n^a,u_n^b) \right\}.
\end{equation}
In other words, the aim of this paper consists of describing the limit energies  in
\eqref{eq:Fnq}, \eqref{eq:Fninfty} and \eqref{eq:Fn0} as $n\to \infty$, when the volumes of
 $\Omega_n^a$ and $\Omega_n^b$ tend to zero under different assumptions on their rate (modelled through the superscript $\ell\in\{q,0,\infty\}$). In particular, according to the values of $\ell$ in \eqref{ell}, we will make specific assumptions on the loads.
 
  \begin{alignat}{2} 
        \hbox{ if } \ell=&
        \,q\in\,\,]0,+\infty[, \qquad && \left\{  \begin{array}{cc}
             H_n^{a,q} \weakstar H& \hbox{ in } L^\infty(\Omega^a;\RI^3), \label{loadq} \\
             H_n^{b, q} \weakstar H& \hbox{ in } L^\infty(\Omega^b;\RI^3),
        \end{array}\right.
        \\
      \hbox{ if }  \ell=&
        +\infty, && \left\{  \begin{array}{cc}
             H_n^{a,\infty} \weakstar H &\hbox{ in } L^\infty(\Omega^a;\RI^3),  \label{loadinfty}\\
             \displaystyle\frac{h_n}{r_n^2}H_n^{b, \infty} \weakstar {\bf 0}& \hbox{ in } L^\infty(\Omega^b;\RI^3),
        \end{array}\right.
        \\
       \hbox{ if } \ell=& 
        0, &&\left\{  \begin{array}{cc}
             \displaystyle\frac{r_n^2}{h_n}H_n^{a,0} \weakstar {\bf 0} &\hbox{ in } L^\infty(\Omega^a;\RI^3),  \\
             H_n^{b,0} \weakstar H & \hbox{ in } L^\infty(\Omega^b;\RI^3).
        \end{array}\right. \label{load0}
    \end{alignat}
where $H\in L^\infty (\O;\mathbb R^{3})$. 

Our analysis will be conducted by means of $\Gamma$-convergence (referring to \cite{DM} for a comprehensive treatment), with respect to the weak * convergence in $BV(\Omega;\mathbb R^3)$, which in view of Propositions 
\ref{Compactnessres1}, \ref{Compactness_infty}, and \ref{Compactness_0},  is the natural one.
We also stress that, due to the growth condition \eqref{coerci} on $W$, there is no loss of generality in replacing this convergence by the $L^1(\Omega;\mathbb R^3)$ - strong convergence. 
In particular we observe that
\begin{Remark}\label{remloads}

    The terms in $F_n^\ell(u_n^a,u_n^b)$ regarding the loads are continuous with respect to $\Gamma$-convergence, i.e. they converge to the same value independently of the approximant sequences $(u_n^a,u_n^b)\in\U_n^\ell$ strongly converging in $L^1$ to $(u^a, u^b)$, if we consider
   the convergences \eqref{loadq}, \eqref{loadinfty}, \eqref{load0}.

    For instance if $\ell=q$, then    $$
    \lim_n\int_{\Omega^a} H_n^{a,q} \cdot u_n^a dx = \int_{\Omega^a} H \cdot u^a dx, \quad \lim_n\int_{\Omega^b} H_n^{b,q} \cdot u_n^b dx = \int_{\Omega^b} H \cdot u^b dx,
    $$
    along any sequences $u_n^a\to u^a$ in $L^1(\Omega^a;\RI^3)$ and $u_n^b\to u^b$ in $L^1(\Omega^b;\RI^3)$.

   For the remaining cases, $\ell=+\infty$ and $\ell=0$ a similar reasoning can be made.

   In view of our latter observation, for which the loads represent continuous terms with respect to $\Gamma$-convergence, we can neglect them in our subsequent analysis.
Clearly the chosen topology to be fully meaningful, it is needed a compactness result for functional bounded sequences. \color{black}
 We underline that,  for the sake of exposition, in the proof of Propositions \ref{Compactnessres1}, \ref{Compactness_infty}, and \ref{Compactness_0}, we consider null loads, referring to the arguments in \cite[Lemma 2.2]{GGLM1} and \cite[Proposition 5.1]{GauZa} for the case of not null loads. We stress that in this case  either we can consider the $L^\infty$ bounds on the loads in accordance with the coercivity constant in \eqref{coerci} in order to get energy bounded sequence, after having exploited Poincar\'e inequality.
   It is also worth to mention that the loads could be considered in divergence form as in \cite[Section 6.2]{FZ} which we refer to, for the Sobolev setting, i.e. 
   \begin{equation*}
   {\rm div}\,{h^{a,\ell}_n}= H^{a,\ell}_n \hbox{ in } \Omega^a \hbox{ and }{\rm div}\,{h^{b,\ell}_n}= H^{b,\ell}_n \hbox{ in } \Omega^b, 
   \end{equation*} 
   under suitable conditions on the regularity of $h^{a,\ell}_n$ and $h^{b,\ell}_n$, and suitable bounds on the $L^\infty$ norm of $h^{a,\ell}_n$ and $h^{b,\ell}_n$, still to be compatible with the coercivity constant of $W$. 
   We leave this case for future work since, by divergence theorem, it will naturally involve the presence of surface forces on the boundary of our multidomain, thus, naturally calling for the analysis of bending effects. On the other hand, if one neglects any surface contribution, the assumption  that loads are in divergence form, will lead naturally to replace terms of the type $\int_{\Omega^a} H^{a,\ell}_n \cdot u^a_n dx$ and  $\int_{\Omega^b} H^{b,\ell}_n \cdot u^b_n dx$ by terms of the form $\int_{\Omega^a} h^{a,\ell}_n \cdot \nabla u^a_n dx$ and $\int_{\Omega^b} h^{b,\ell}_n \cdot \nabla u^b_n dx$, respectively, for which we still have convergence by duality.

For the $\Gamma$- convergence in itself, one can clearly start the asymptotic analysis considering $L^1$ strong converging sequences $(u_n^a)_n$ and $(u_n^b)_n$. \color{black}

\end{Remark}

The following result summarizes the content of this paper. We refer to Section \ref{pre} for the notation adopted in the theorem to represent the derivatives of fields with Bounded Variation. 
We also underline that, with an abuse of notation we identify fields $u^a \in BV(\Omega^a;\mathbb R^3)$ such that $D_\alpha u^a=({\bf 0}, {\bf 0})\in \mathbb R^{3 \times 2}$, with fields $u^a \in BV(]0,1[;\mathbb R^3)$. Analogously we identify $u^b \in BV(\Omega^b;\mathbb R^3)$ such that $D_3 u^b={\bf 0} \in \mathbb R^3$, with fields in $BV(\omega;\mathbb R^3)$.

\begin{Theorem}
    \label{generalrepresentation}
    Let $W\;:\; \RI^{3\x3}\to \RI$ be a Borel function satisfying \eqref{coerci}, and $F_n^\ell$ be as in Definition \ref{def:energies} depending on the superscript $\ell$ defined in \eqref{ell}. Denote by $\hat W:{\mathbb R}^3\to \mathbb R$ the function  given by

    \begin{align}\label{hatW}
        \hat W(\xi_3):=\inf_{\xi_\alpha \in \mathbb R^{3\times 2}} W(\xi_\alpha| \xi_3),
    \end{align}  
    and let $W_0:\mathbb R^{3\times 2}\to \mathbb R$ be the function defined 
    \begin{align}\label{W0}
        W_0(\xi_\alpha):=\inf_{\xi_3 \in \mathbb R^{3}} W(\xi_\alpha| \xi_3).
    \end{align}
    
    \noindent Assume that the recession functions of $\hat W$ and $W_0$, 
     $\hat W^\infty:\mathbb R^{3}\to \mathbb R$, and $W_0^
\infty: \mathbb R^{3\times 2} \to \mathbb R$, defined as
                \begin{equation}\label{S101rec}
                        \hat W^\infty(\xi):=\limsup_{ t\to+\infty} \frac{\hat W(t\xi)}{t}, \hbox{ and } W_0^\infty(\xi):= \limsup_{t \to +\infty}\frac{W_0(t \xi)}{t},
                \end{equation}
                respectively, 
 satisfy the following: there exist $C>0,  0<m<1, L>0$ such that
    \begin{align}\label{Whatrec}
        \left|(\hat W)^\infty(\xi_3)- \frac{\hat W(t\xi_3)}{t}\right| \leq \frac{C}{t^m}
    \end{align}
        for $\xi_3 \in \mathbb R^{3}, \|\xi_3\|=1$, $t >L$,
        and
        \begin{align}\label{W0rec}
        \left|W_0^\infty(\xi_\alpha)- \frac{W_0(t\xi_\alpha)}{t}\right| \leq \frac{C}{t^m}
        \end{align}
        for $\xi_\alpha \in \mathbb R^{3\times 2}, \|\xi_\alpha\|=1$, $t >L$.
    Let \(\mathcal{U}_n^\ell\) be the space in \eqref{Vn}. 
    \noindent Define further $K^{a,{\ell}}:BV(]0,1[;\mathbb R^3)\to\RI$ and $K^{b,{\ell}}:BV(\o;\mathbb R^3)\to\RI$ as
    \begin{align}
        K^{a,\ell}(u^a):=
        \displaystyle{\int_{]0,1[}{\hat W}^{**}\left(
        \nabla_{3}u^a(x_3)\right)}dx_3 &+ \displaystyle{\int_{]0,1[}({\hat W}^{**})^{\infty}\left(\frac{d D^s_3 u^a}{d |D^s_3 u^a|},\right)}d |D^s_3 u^a| 
        \nonumber
        \\
        \label{def:Ka}
        &+ (\hat W^{\ast \ast})^\infty (c^{a,\ell}- u^a(1^-))
        +\int_{\Omega^a} H(x)\cdot u^a(x_3) dx, 
    \end{align}
    \begin{align}
        K^{b,\ell}(u^b):=
        &\displaystyle{\int_{\omega}Q W_0\left(\nabla_{\alpha }u^b(x_\alpha)\right)dx_\alpha} +\int_{\omega}(QW_0)^{\infty}\left(\frac{D^s_\alpha u^b}{|D^s_\alpha u^b|}\right)d |D^s_{\alpha}u^b|+
        \nonumber
        \\
        &+\int_{\partial \omega \times ]-1,0[} (QW_0)^\infty((f^{b,\ell}-(u^b)^-)\otimes \nu_{\partial \omega \times ]-1,0[})d \mathcal H^2 + \int_{\Omega^b} H(x)\cdot u^b(x_\alpha) dx,
        \label{def:Kb}
    \end{align}
    where  $u^a(1^-)$ represent the internal trace of $u^a$, $Q (W_0)$ and $\hat W^{\ast \ast}$ represent the quasiconvex and the convex envelopes (see subsection \ref{convnot}) of $W_0$ and $\hat W$ in \eqref{W0} and \eqref{hatW}, respectively, and the superscript $\infty$ stands for their recession functions in the sense of \eqref{S101rec}.

    Then the following representations hold:

    \vspace{5mm}
    \textbf{Case $\ell=q$}:
    Assume that $q\in]0,+\infty[$, that the Dirichlet conditions in $\U_n^q$ are as \eqref{dirichletq} and the loads satisfy \eqref{loadq}. Then for every  $(u^a, u^b) \in BV(]0,1[;\mathbb R^3) \x BV(\omega;\mathbb R^3)$, 
    \begin{align}\label{probl_q}
        \inf\Big\{&\liminf_n F_n^q(u_n^a,u_n^b): (u_n^a, u_n^b)\in \mathcal U_n^q,
        \nonumber\\
        &u_n^a \weakstar u^a \hbox{ in }BV(\Omega^a;\mathbb R^3), 
        u_n^b  \weakstar u^b \hbox{ in }BV(\Omega^b;\mathbb R^3)\Big\}
        = K^{a,q}(u^a) + q K^{b,q}(u^b).
    \end{align}

    \vspace{5mm}
    \textbf{Case $\ell=+\infty$}:
Assume 
that the Dirichlet conditions in $\U_n^\infty$ are as in \eqref{dirichletinfty} and the loads satisfy \eqref{loadinfty}. Assume further that there exists $C>0$ such that  
    \begin{equation}\label{eq:growth1}
\frac{1}{C}|M-Id|\leq W(M), \quad W(Id)=0,
    \end{equation}
    for every $M \in \mathbb R^{3\times 3}$, with $Id$ being the identity matrix in $\mathbb R^{3\times 3}$.
  
    We consider two different subcases.

    \begin{itemize}
        \item[--]
            If $\displaystyle\frac{h_n}{r_n} \to 0$, then for every  $u^a \in BV(]0,1[;\mathbb R^3)$,
            \begin{align}
                \inf\Big\{&\liminf_n F_n^\infty(u_n^a,u_n^b): (u_n^a, u_n^b)\in \mathcal U_n^\infty,  
                \nonumber\\
                &u_n^a \weakstar u^a \hbox{ in }BV(\Omega^a;\mathbb R^3), 
                u_n^b  \to (x_\alpha,0) \hbox{ in }W^{1,1}(\Omega^b;\mathbb R^3)\Big\}
                = K^{a,\infty}(u^a) \label{probl_infty,0}.
            \end{align}

        \item [--]
            If $\displaystyle\frac{h_n}{r_n} \to \infty$ then, for every  $u^a \in BV(]0,1[;\mathbb R^3) $, 
            \begin{align}
                \inf\Big\{&\liminf_n F_n^\infty(u_n^a,u_n^b): (u_n^a, u_n^b)\in \mathcal U_n^\infty,  
                u_n^a \weakstar u^a \hbox{ in }BV(\Omega^a;\mathbb R^3), 
                \nonumber
                \\
                &u_n^b  \to (x_\alpha,0) \hbox{ in }W^{1,1}(\Omega^b;\mathbb R^3)\Big\}
                = K^{a,\infty}(u^a) + (\hat W^{\ast \ast})^\infty (u^a(0^+)).
                \label{probl_infty,infty}
            \end{align}
            where $u^a(0^+)$ is the  internal trace of the function $u^a$ at $0$. 
    \end{itemize}

    \vspace{5mm}
    \textbf{Case $\ell=0$}:
      Assume further  
      that the Dirichlet conditions in $\U_n^0$ are as in \eqref{dirichlet0} and the loads satisfy \eqref{load0}. Assume further \eqref{eq:growth1}. Then, for every  $u^b \in BV(\omega;\mathbb R^3)$, 
    \begin{align}\label{prob_0}
        \inf\Big\{&\liminf_n F_n^0(u_n^a,u_n^b): (u_n^a, u_n^b)\in \mathcal U_n^o,
        \nonumber\\
        &u_n^a \to (0_\alpha,x_3) \hbox{ in }W^{1,1}(\Omega^a;\mathbb R^3), 
        u_n^b  \weakstar u^b \hbox{ in }BV(\Omega^b;\mathbb R^3)\Big\}
        = K^{b,0}(u^b).
    \end{align}

\end{Theorem}


\begin{Remark}\label{refremweak}

We underline that our analysis could be performed, without requiring  \eqref{W0rec} and \eqref{Whatrec}, presenting other arguments than those in \cite{BFMTraces}, when relaxing from trace constrained Sobolev spaces to $BV$ ones, since our density $W$ does not depend on the potentials $u_n$'s, see \cite{FMZ} for detailed proofs ( also in the second order setting) relying on suitable formulations of Reshetnyak's type theorems for the convergence of measure dependent energies. 
In particular \eqref{W0rec}
 and \eqref{Whatrec} ensure the existence of the so called {\it strong recession functions} (see \cite{KR}) for $Q\hat W$ and $W_0^{\ast \ast}.$

 It is worth to emphasize that, in view of Propositions \ref{Compactnessres1}, \ref{Compactness_infty} and \ref{Compactness_0}, the strong convergences appearing in the definitions of problems \eqref{probl_infty,0}, \eqref{probl_infty,infty}  and \eqref{prob_0} are redundant. We could still consider only the $BV$ weak* convergence as in the problem \eqref{probl_q}.

We also underline that in most of the cases $\ell \in [0,+\infty]$, the limiting energies are  decoupled. In particular, due to the choices of the boundary conditions and to technical assumption on the minimum of $W$ expressed by \eqref{eq:growth1} which forces certain energetic responses, in the  cases $\ell=\infty$ and $\ell=0$ the sample is fixed below and above  respectively, but if $\ell=+\infty$ and $\lim_{n}\frac{h_n}{r_n}=\infty$, then  an energetic penalization appears at the junction point $\bf 0$. Namely, while in the case $\ell=+\infty$ and $\lim_{n}\frac{h_n}{r_n}=0$, the infimizing sequences in the subdomain below can "stretch" near the joining area $r_n\omega\x \{0\}$ and still keep their energy infinitesimal, in the case where  $h_n$ tends "much slower" to $0$ than $r_n^2$ ( i.e. $h_n/r_n\to\infty$) than this stretching does not allow any value because $u_n^b$ is approaching the rigid limit function in such a manner that its values $r_n\omega\x\{0\}$ are in fact tending to ${\bf 0}$ thus leading to the aforementioned penalization.

We also observe that in the some parts of the proof we consider $W$ continuous (Lipschitz), since, in view of Lemma \ref{lemma29}, it can be assumed to be quasiconvex (see \cite{DA}). 
\end{Remark}

\section{Preliminary results}\label{pre}
In this section we fix notation, we collect the main properties of  the space of BV functions that we need to prove our results, recall some convexity notions and prove some preliminary results.

We will use the following notations
\begin{itemize}
\item[-] $O \subset \mathbb R^{h}$ is an open set;
\item[-] ${\mathcal A}(O)$ is the family of all open subsets
of $O $;
\item [-] $\bf 0$ and $0_\alpha$ denote the null vectors in $\mathbb R^3$ and $\mathbb R^2$, respectively;
\item [-] $C^\infty_c(O;\mathbb R^k)\coloneqq\{u\colon U\to\mathbb R^k: \text{$u$ is smooth and has compact support in $O$}\}$; if $k=1$, we just denote this set by $C^\infty_c(O)$;
$C^\infty_0(O)$ and $C^\infty_0(O;\mathbb R^k)$ denote the closures of $C^\infty_c(O)$ and $C^\infty_c(O;\mathbb R^k)$, respectively, in the $\sup$ norm; 
\item[-] $\mathcal M (O)$ and $\mathcal M(O;\mathbb R^{k})$ are the sets of (signed) finite real-valued or vector-valued Radon measures on $O$, respectively; $\mathcal M ^+(O)$ is the set of non-negative finite Radon measures on $O$;
\item[-] given $\mu\in\mathcal M(O)$ or $\mu\in\mathcal M(O;\mathbb R^{k})$, the measure $|\mu|\in\mathcal M^+(O)$ denotes the total variation of $\mu$;
\item[-] $\mathcal L^{h}$ and $\mathcal H^{h-1}$ denote the  $h$-dimensional Lebesgue measure and the $\left(  h-1\right)$-dimensional Hausdorff measure in $\mathbb R^ h$, respectively; the symbol $ dx$ will also be used to denote integration with respect to $\mathcal L^{h}$, while $d {\mathcal H}^{h-1}$ will be used to denote surface integration with respect to $\mathcal H^{h-1}$;

\item[-] $\mathcal S^{h-1}$ denotes the unit sphere in $\mathbb R ^h$;

\item[-] $Q$ will denote the open
unit cube in $\mathbb R^h$ i.e. $Q :=]0, 1[^h$ and, given $\nu \in \mathcal S^{h-1}$, $Q_\nu$ denotes 
some open unit cube with two faces perpendicular to $\nu$.
\item[-] for $x \in \mathbb R^{h}, Q(x, \varepsilon) := x + \varepsilon Q$ and similarly, $Q_\nu (x; \varepsilon) := x + \varepsilon Q_\nu$;
\item[-] $B$ will denote the open unit ball in $\mathbb R^h$ centered at the origin and $B(x, \varepsilon) := x + \varepsilon B$;
\item[-] $\overline B$ will denote the closed unit ball centered at the origin and $\overline B (x, \varepsilon) := x + \varepsilon\overline B$;
\item[-] $\langle \cdot, \cdot \rangle$ denotes the Euclidean scalar product, which will be clear from the context in which vector space acts.
\item[-] where not otherwise specified, $C$ denotes a positive constant which may vary from line to line;

\end{itemize}

\begin{Definition}
 A function $w\in L^{1}(O;{\mathbb{R}}^{k})$ is said to be of
\emph{bounded variation}, and we write $w\in BV(O;{\mathbb{R}}^{k})$, if
all its first distributional derivatives $D_{j}w^{i}$ belong to $\mathcal{M}%
(\Omega)$ for $1\leq i\leq k$ and $1\leq j\leq h$.
\end{Definition}

The matrix-valued measure whose entries are $D_{j}w_{i}$ is denoted by $Dw$
and $|Dw|$ stands for its total variation.
We observe that if $w\in BV(O;\mathbb{R}^{k})$ then $w\mapsto|Dw|(O)$ is lower
semicontinuous in $BV(O;\mathbb{R}^{k})$ with respect to the
$L_{\mathrm{loc}}^{1}(O;\mathbb{R}^{k})$ topology.

We briefly recall some facts about functions of bounded variation. For more
details we refer the reader to \cite{AFP, Ziemer}.

\begin{Definition}
Given $w\in BV\left( O;\mathbb{R}^{k}\right)  $ the \emph{approximate
upper}\textit{\ }\emph{limit }and the \emph{approximate lower limit} of each
component $w^{i}$, $i=1,\dots,k$, are defined by
\[
\left(  w^{i}\right)  ^{+}\left(  x\right)  :=\inf\left\{  t\in\mathbb{R}%
:\,\lim_{\varepsilon\rightarrow0^{+}}\frac{\mathcal{L}^{h}\left(  \left\{
y\in O\cap Q\left(  x,\varepsilon\right)  :\,w^{i}\left(  y\right)
>t\right\}  \right)  }{\varepsilon^{h}}=0\right\}
\]
and
\[
\left(  w^{i}\right)  ^{-}\left(  x\right)  :=\sup\left\{  t\in\mathbb{R}%
:\,\lim_{\varepsilon\rightarrow0^{+}}\frac{\mathcal{L}^{h}\left(  \left\{
y\in O\cap Q\left(  x,\varepsilon\right)  :\,w^{i}\left(  y\right)
<t\right\}  \right)  }{\varepsilon^{h}}=0\right\}  ,
\]
respectively. The \emph{jump set}\textit{\ }of $w$ is given by
\[
J_{w}:=\bigcup_{i=1}^{k}\left\{  x\in O:\,\left(  w^{i}\right)
^{-}\left(  x\right)  <\left(  w^{i}\right)  ^{+}\left(  x\right)  \right\}
.
\]

\end{Definition}

It can be shown that $J_{w}$ and the complement of the set of Lebesgue points
of $w$ differ, at most, by a set of $\mathcal{H}^{h-1}$ measure zero.
Moreover, $J_{w}$ is $\left(  h-1\right)  $-rectifiable, i.e., there are
$C^{1} $ hypersurfaces $\Gamma_{i}$ such that 
$\mathcal{H}^{h-1}\left(  J_{w}\setminus\cup_{i=1}^{\infty}\Gamma_{i}\right)=0.$

\begin{Proposition}\label{thm2.3BBBF}
If $w\in BV\left( O;\mathbb{R}^{k}\right)  $ then

\begin{enumerate}
\item[i)] for $\mathcal{L}^{h}-$a.e. $x\in O$%
\begin{equation*}
\lim_{\varepsilon\rightarrow0^{+}}\frac{1}{\varepsilon}\left\{  \frac
{1}{\mathcal{\varepsilon}^{h}}\int_{Q\left(  x,\varepsilon\right)
}\left\vert w(y)  -w(x)  -\nabla w\left(
x\right)  \cdot( y-x)  \right\vert ^{\frac{h}{h-1}%
}dy\right\}  ^{\frac{h-1}{h}}=0; 
\end{equation*}

\item[ii)] for $\mathcal{H}^{h-1}$-a.e. $x\in J_{w}$ there exist $w^{+}\left(
x\right)  ,$ $w^{-}\left(  x\right)  \in\mathbb{R}^{k}$ and $\nu\left(
x\right)  \in S^{h-1}$ normal to $J_{w}$ at $x,$ such that
\[
\lim_{\varepsilon\rightarrow0^{+}}\frac{1}{\varepsilon^{h}}\int_{Q_{\nu}%
^{+}\left(  x,\varepsilon\right)  }\left\vert w\left(  y\right)  -w^{+}\left(
x\right)  \right\vert dy=0,\qquad\lim_{\varepsilon\rightarrow0^{+}}\frac
{1}{\varepsilon^{h}}\int_{Q_{\nu}^{-}\left(  x,\varepsilon\right)  }\left\vert
w\left(  y\right)  -w^{-}\left(  x\right)  \right\vert dy=0,
\]

where $Q_{\nu}^{+}\left(  x,\varepsilon\right)  :=\left\{  y\in Q_{\nu}\left(
x,\varepsilon\right)  :\,\left\langle y-x,\nu\right\rangle >0\right\}  $ and
$Q_{\nu}^{-}\left(  x,\varepsilon\right)  :=\left\{  y\in Q_{\nu}\left(
x,\varepsilon\right)  :\,\left\langle y-x,\nu\right\rangle <0\right\}  $;

\item[iii)] for $\mathcal{H}^{h-1}$-a.e. $x\in O\backslash J_{w}$%
\[
\lim_{\varepsilon\rightarrow0^{+}}\frac{1}{\mathcal{\varepsilon}^{h}}%
\int_{Q\left(  x,\varepsilon\right)  }\left\vert w(y)
-w\left(  x\right)  \right\vert dy=0.
\]

\end{enumerate}
\end{Proposition}

We observe that in the vector-valued case in general $\left(  w^{i}\right)
^{\pm}\neq\left(  w^{\pm}\right)  ^{i}.$ In the sequel $w^{+}$ and $w^{-}$
denote the vectors introduced in $ii)$ above.

Choosing a normal $\nu_{w}\left(  x\right)  $ to $J_{w}$ at $x,$ we denote
the \emph{jump} of $w$ across $J_{w}$ by $\left[  w\right]  :=w^{+}-w^{-}.$
The distributional derivative of $w\in BV\left(  O;\mathbb{R}^{k}\right)
$ admits the decomposition
\[
Dw=\nabla w\mathcal{L}^{h}\lfloor O+\left(  \left[  w\right]  \otimes
\nu_{w}\right)  \mathcal{H}^{h-1}\lfloor J_{w}+D^{c} w,
\]
where $\nabla w$ represents the density of the absolutely continuous part of
the Radon measure $Dw$ with respect to the Lebesgue measure. The
\emph{Hausdorff}, or \emph{jump}, \emph{part} of $Dw$ is represented by
$\left(  \left[  w\right]  \otimes\nu_{w}\right)  \mathcal{H}^{h-1}\lfloor
J_{w}$ and $D^{c} w $ is the \emph{Cantor part} of $Dw$. The measure $D^{c} w
$ is singular with respect to the Lebesgue measure and it is diffuse, i.e.,
every Borel set $B\subset O$ with $\mathcal{H}^{h-1}\left(  B\right)
<\infty$ has Cantor measure zero.

The following result, that will be exploited in the sequel, can be found in \cite[Lemma 2.6]{FM2}.
\begin{Lemma}\label{lemma2.5BBBF}
Let $w \in BV( O;\mathbb R^k)$, for ${\cal H}^{h-1}$ a.e. $x$ in $J_w$,
$$
\displaystyle{\lim_{\e \to 0^+} \frac{1}{\e^{h-1}} \int_{J_w \cap  Q_{\nu(x)}(x, \e)} |w^+(y)- w^-(y)| d {\cal H}^{h-1} = |w^+(x)- w^-(x)|.}
$$
\end{Lemma}

We state here the compactness result in $BV$ (Proposition $3.23$ in \cite{AFP}).
\begin{Proposition}\label{compBV}
    Every sequence $(u_n)_n \subset BV_{\text{loc}}( O; \mathbb R^k)$ satisfying
    $$ \sup \left\{ \int_A |u_n| \; dx + |Du_n|(A),\; n \in \mathbb N \right\} < +\infty \qquad \forall A \in \mathcal A (O), \;A \subset \subset O $$
    admits a subsequence $(u_{n(k)})_k$ converging in $L^1_{\text{loc}}(O; \mathbb R ^k)$ to $u \in BV_{\text{loc}}(O; \mathbb R^k).$ If $O$ is a bounded extension domain and the sequence is bounded in $BV(O; \mathbb R^k)$ we can say that $u \in BV(O; \mathbb R^k)$ and the subsequence weakly* converges to $u$.
\end{Proposition}
Finally, we recall that any open set $O$ with compact Lipschitz boundary is an extension domain (Proposition $3.21$ in \cite{AFP}).

\subsection{Convexity notions}\label{convnot}

In order to obtain the representation result we need to recall some well known
notions of convex analysis, as well as some more general properties, essentially
concerning  convexity and quasiconvexity. Let 
$g : \mathbb R^h \to \mathbb R$ be a function, the convex envelope (or \textit{convexification})
of $g$ is the function
\begin{align*}
g^{\ast \ast}: = \sup\{h \leq g : h \hbox{ convex} \}. 
\end{align*}

\begin{Definition}
A Borel function $f : \mathbb R^{k\times h} \to \mathbb R$ is said to be quasiconvex at  $v$ if
$$f(v) \leq \int_Q f(v + \nabla w(x))dx$$ for every $w \in W^{1,\infty}_0(\mathbb R^h;\mathbb R^k).$
\end{Definition}

\begin{Definition}\label{quasiconvexificationdef} Given a Borel function $f : \mathbb R^{k\times h}\to \mathbb R$, the quasiconvexification
of $f$ at $v \in \mathbb R^{k\times h}$ is given by
$$
Qf(v) := \inf\left\{
\int_Q
f(v + \nabla w(x))dx : w \in C^\infty_0 (Q;\mathbb R^k)\right\}.$$
\end{Definition}
The above Definition  allows to recover the usual notion of quasiconvex envelope, i.e. the greatest quasiconvex function below $f$, see \cite{DA}. 

 The following result allows us to prove that there is no loss of generality in assuming $W$ quasiconvex, hence continuous in our main theorem.

\begin{Lemma}\label{lemma29}
    Let $W\;:\; \RI^{3\x3}\to \RI$ be a Borel function satisfying \eqref{coerci}, and $F_n^\ell$ be as in Definition \ref{def:energies} depending on the superscript $\ell$ defined in \eqref{ell}. Let also $\overline{F}_n^\ell$ be defined as $F_n^\ell$ but by replacing $K_n^{a,\ell}$ and $K_n^{b,\ell}$ respectively with:
    $$
     \overline{K_n^{a,\ell}}( u^a_n) = \int_{{\Omega}^a}
\left( QW\left(\displaystyle{\frac{1}{r_n}\nabla_{\alpha}u^{a}_n(x)} \;\; | \;\;
\displaystyle{\nabla_{3}u^{a}_n}(x)\right)+ H_n^{a,\ell}(x) \cdot u^a_n(x)\right)dx,
    $$
    $$
\overline{K_n^{b,\ell}}( u^b_n):= \int_{{\Omega}^b}
\left(QW\left(\displaystyle{\nabla_{\alpha}u^b_n(x)} \;\; | \; \;
\displaystyle{\frac{1}{h_n}\nabla_{3}u^b_n(x)}
\right)+ H_n^{b,\ell}(x)\cdot u_n^b(x)\right)dx.    
    $$
Here $QW$ stands for the quasiconvexification of the integrand $W$.

Let 
\begin{align}
\nonumber
\mathcal R^\ell:=\Big\{ &(u^a, u^b)\in BV(\Omega^a;\mathbb R^3)\times BV(\Omega^b;\mathbb R^3): \exists (u^a_n,u^b_n)\in \mathcal U_n^\ell;
\\
\label{def:Rell}
 &(u^a_n, u^b_n)\overset{\ast}{\rightharpoonup}(u^a,u^b)\in BV(\Omega^a;\mathbb R^3)\times BV(\Omega^b;\mathbb R^3)
 \\
 \nonumber
  &\hbox{ and } \limsup_n F_n^\ell(u^a_n, u^b_n) < +\infty\Big\}.
 \end{align}

Then, the following functionals $G^\ell$, $\overline {G^\ell}: \mathcal R^\ell \to \mathbb R$ are equal:

    \begin{align}
        \label{normal_inf}
        G^\ell(u^a,u^b)&:=\inf\left\{\liminf_{n\to \infty} F_n^\ell(u_n^a,u_n^b): (u_n^a, u_n^b)\in \mathcal U_n^\ell,  \right. \nonumber\\
&\left.
u_n^a \overset{\ast}{\rightharpoonup} u^a \hbox{ in }BV(\Omega^a;\mathbb R^3), 
u_n^b  \overset{\ast}{\rightharpoonup} u^b \hbox{ in }BV(\Omega^b;\mathbb R^3)\right\},
\\
\nonumber
\\
\label{quasiconv_inf}
\overline{G^\ell}(u^a,u^b)&:=\inf
\left
\{\liminf_{n\to \infty}\overline{F}_n^\ell(u_n^a,u_n^b): (u_n^a, u_n^b)\in \mathcal U_n^\ell,  \right. \nonumber\\
& \left.
u_n^a \overset{\ast}{\rightharpoonup} u^a \hbox{ in }BV(\Omega^a;\mathbb R^3), 
u_n^b  \weakstar u^b \hbox{ in }BV(\Omega^b;\mathbb R^3)
\right
\}.
 \end{align}
\end{Lemma}
\begin{proof}

    From the inequality $QW\leq W$ it easily follows that $\overline{G}^\ell\leq G^\ell$.
    Now we prove the other inequality.  Let $(u^a,u^b)\in\R^\ell$.  There is no loss of generality in assuming $\overline{G}^\ell(u^a, u^b)<+\infty$.

    As already observed above, the contribution of the loads can be neglected since it is continuous with respect to $\Gamma$-convergence. 

Let $(u_n^a,u_n^b) \subset \mathcal{U}^\ell_n$ be an infimizing sequence of \eqref{quasiconv_inf}, namely
    $$
    \lim_n\overline{F}^\ell_n(u_n^a,u_n^b) = \overline{G}^\ell(u^a,u^b)<+\infty.
    $$
    
    For any $n\in \N$, from Theorem $9.1$ in \cite{DA},  we have that
    
    \begin{align}
        \int_{\Omega^a} QW_n^a(\nabla u_n^a(x)) dx = \inf \left\{\liminf_{k\to\infty} \int_{\Omega^a} W_n^a(\nabla u_k(x))dx: \; u_k \in u_n^a + W_0^{1,1}(\Omega^a;\RI^3)
        \right.
        \nonumber
        \\
        \left.
        u_k \to u_n^a \hbox{ in } L^1(\Omega^a;\RI^3)
        \right\},
        \label{relax_daco_a}
        \\
        \nonumber
        \\
        \int_{\Omega^b} QW_n^b(\nabla u_n^b(x)) dx = \inf \left\{\liminf_{k\to\infty} \int_{\Omega^b} W_n^b(\nabla u_k(x))dx: \; u_k \in u_n^b + W_0^{1,1}(\Omega^b;\RI^3)
        \right.
        \nonumber
        \\
        \left.
        u_k \to u_n^b \hbox{ in } L^1(\Omega^b;\RI^3)
        \right\},
        \label{relax_daco_b}
    \end{align}
where we define,  for every $\xi \in \mathbb R^{3 \times 3}$,

$$
W_n^a(\xi) := W\left(  \frac{1}{r_n}\xi_\alpha  \;|\;  \xi_3\right), \quad  W_n^b(\xi) := W\left(  \xi_\alpha  \;|\;  \frac{1}{h_n}\xi_3\right).
$$
Notice that
$$
QW^a_n(\xi) = QW\left(  \frac{1}{r_n}\xi_\alpha  \;|\;  \xi_3\right), \quad QW_n^b(\xi) = QW\left(  \xi_\alpha  \;|\;  \frac{1}{h_n}\xi_3\right),
$$
for every $\xi \in \mathbb R^{3 \times 3}$,
(see \cite{BFMbend} for the details of similar equalities). 
Based on \eqref{relax_daco_a} and \eqref{relax_daco_b} we can therefore choose for fixed $n\in\N$, $u^a_{n,k} \in u_n^a + W_0^{1,1}(\Omega^a;\RI^3)$ and $u^b_{n,k} \in u_n^b + W_0^{1,1}(\Omega^b;\RI^3)$ such that 
\begin{align*}
    u^a_{n,k} \to u_n^a \hbox{ in } L^{1}(\Omega^a;\RI^3), &\quad u^b_{n,k} \to u_n^b \hbox{ in } L^{1}(\Omega^b;\RI^3) \hbox{ as }k \to \infty, 
    \\
    \lim_k \int_{\Omega^a} W\left(  \tfrac{1}{r_n}\nabla_\alpha u_{n,k}^a(x) \;|\;  \nabla_3u_{n,k}^a(x)\right)dx &= \int_{\Omega^a} QW\left(  \tfrac{1}{r_n}\nabla_\alpha u_{n}^a(x) \;|\;  \nabla_3u_{n}^a(x)\right) dx,
    \\
    \lim_k \int_{\Omega^a} W\left(  \nabla_\alpha u_{n,k}^b(x) \;|\;  \tfrac{1}{h_n}\nabla_3u_{n,k}^b(x)\right)dx &= \int_{\Omega^a} QW\left(  \nabla_\alpha u_{n}^b(x) \;|\;  \tfrac{1}{h_n}\nabla_3u_{n}^b(x)\right) dx.
\end{align*}

Since the traces of $u_{n,k}^a$ and $u_{n,k}^b$ are respectively the traces of $u_n^a$ and $u_n^b$ it results that $(u_{n,k}^a$,$u_{n,k}^b) \in \mathcal{U}^\ell_n$. By a diagonal argument we can construct sequences $(u_{n,k(n)}^a,u_{n,k(n)}^b)\in \mathcal{U}^\ell_n$ such that:
\begin{align}
\label{conv_diag_lemma}
 u^a_{n,k(n)} \to u^a \hbox{ in } L^{1}(\Omega^a;\RI^3) &\quad u^b_{n,k(n)} \to u^b \hbox{ in } L^{1}(\Omega^b;\RI^3) \hbox{ as }n \to \infty,
 \\
 \left| F_n^\ell(u_{n,k(n)}^a,u_{n,k(n)}^b) \;\right.&-\left.\; \overline{F_n^\ell}(u_{n}^a,u_{n}^b) \right| \leq 1/n. \nonumber
\end{align}

Therefore
\begin{equation}
\label{recoveryquasiconvex}
\lim_{n\to \infty} F_n^\ell(u_{n,k(n)}^a,u_{n,k(n)}^b) = \lim_{n \to \infty}\overline{F^\ell_n}(u_n^a,u_n^b) = \overline{G^\ell}(u^a,u^b) < +\infty.
\end{equation}

Recall that $\ell$ is defined as in \eqref{ell}. From the above inequality and \eqref{coerci} we can easily see that for any case $\ell\in [0,+\infty]$, 
we have that 
$$
\sup_{n\in \mathbb N}\|\nabla u_{n,k(n)}^a \|_{L^1(\Omega^a;\RI^3)} < +\infty, \quad \sup_{n\in \mathbb N}\|\nabla u_{n,k(n)}^b \|_{L^1(\Omega^b;\RI^3)} < +\infty.
$$
From \eqref{conv_diag_lemma} it follows
$$
u_{n,k(n)}^a  \overset{\ast}{\rightharpoonup} u^a \hbox{ in }BV(\Omega^a;\mathbb R^3) \quad u_{n,k(n)}^b  \overset{\ast}{\rightharpoonup} u^b \hbox{ in }BV(\Omega^b;\mathbb R^3), \hbox{ as }n \to \infty.
$$
We stress that our arguments cover all the cases for $\ell$. Indeed, as already done in Remark \ref{refremweak}, the strong convergences appearing for the scaled gradients in the definitions of the $\Gamma$-limits in \eqref{probl_infty,0}, \eqref{probl_infty,infty},  and \eqref{prob_0} are redundant (see Proposition \ref{Compactnessres1} and \ref{Compactness_infty})  in the sense that it will be enough to assume that there is weak* convergence in $BV$, then the strong will follow by \eqref{coerci} and the imposed boundary conditions. 
Since the sequences $(u_{n,k(n)}^a,u_{n,k(n)}^b)$ are admissible to problem \eqref{normal_inf}, from \eqref{recoveryquasiconvex} we conclude $G^\ell(u^a,u^b) \leq \overline{G}^\ell(u^a,u^b)$.
\end{proof}

The following result is a consequence of standard diagonalization arguments, the coercivity assumption in \eqref{coerci}, the metrizability of weak* topology in $BV$ on compact sets,  and the attainment of the functional $G^\ell$ in \eqref{normal_inf}, by a subsequence of $(F_n^\ell(u^a_n, u^b_n))$ with $(u^a_n, u^b_n) \in \mathcal U_n^\ell$. 
Once again, we emphasize that the arguments below cover all the energies appearing in Theorem \ref{generalrepresentation}.

\begin{Lemma}
\label{preFlsc} 
Let $W\;:\; \RI^{3\x3}\to \RI$  be a Borel  function satisfying \eqref{coerci}, and $F_n^\ell$ be as in Definition \ref{def:energies} depending on the superscript $\ell$ defined in \eqref{ell}. Let $G^\ell$ be the functional defined in \eqref{normal_inf}. Then it is lower semicontinuous on the space $\mathcal R^\ell$, in \eqref{def:Rell}, with respect to the $BV(\Omega^a;\mathbb R^3)\times BV(\Omega^b;\mathbb R^3)$ weak * convergence.
\end{Lemma}
\begin{proof}
 As observed above, the lower semicontinuity result can be proven with respect to weak* convergence in $BV$ in all the cases $\ell \in [0,+\infty]$, since the compactness results will provide autonomously stronger convergences in the cases $\ell\in\{0,+\infty\}$.   Let $(u^a,u^b)\in\R^\ell$ and $(u^a_n,u^b_n)\in\R^\ell$ such that $u^a_n\weakstar u^a$ in $BV(\Omega^a;\RI^3)$, $u^b_n\weakstar u^b$ in $BV(\Omega^b;\RI^3)$.
    Let $(u_{n,k}^a,u_{n,k}^b)\in\U_k^\ell$ be sequences such that 
    \begin{align*}
        u^a_{n,k} \weakstar u_n^a \hbox{ in } BV(\Omega^a;\RI^3), &\quad u^b_{n,k} \weakstar u_n^b \hbox{ in } BV(\Omega^b;\RI^3) \quad \hbox{as }k\to\infty,
        \\
        \lim_k F_k^\ell(u_{n,k}^a,u_{n,k}^b) &= G^\ell(u_n^a,u_n^b).
    \end{align*}

    By a diagonal argument we can construct a sequence $(u_{n,k(n)}^a,u_{n,k(n)}^b) \in \U_{k(n)}^\ell$ for $n\in\N$ such that

    \begin{align*}
        u^a_{n,k(n)} \weakstar u^a \hbox{ in } BV(\Omega^a;\RI^3), &\quad u^b_{n,k(n)} \weakstar u^b \hbox{ in } BV(\Omega^b;\RI^3),
        \\
        \left|F_{k(n)}^\ell(u_{n,k(n)}^a,u_{n,k(n)}^b) \; \right.
        &-
        \left.
        \; G^\ell(u_n^a,u_n^b) \right| \leq \frac{1}{n}.
    \end{align*}

    The sequences $(u_{n,k(n)}^a,u_{n,k(n)}^b)$ are admissible for the infimizing problem $G^\ell(u^a,u^b)$ in \eqref{normal_inf} and from the above inequality
    $$
    \liminf_{n\to\infty} G^\ell(u_{n}^a,u_{n}^b)=\liminf_{n\to \infty} F_{k(n)}^\ell(u_{n,k(n)}^a,u_{n,k(n)}^b) \geq G^\ell(u^a,u^b).
    $$
 \end{proof}

\section{Main results}\label{MR}

We are dividing the section in three subsections according to the cases \eqref{eq:Fnq}, \eqref{eq:Fninfty} and \eqref{eq:Fn0}. 

In each subsection we detail the compactness result, and prove the lower bound and the upper bound inequalities that lead to the statement in Theorem \ref{generalrepresentation}.

\subsection{Finite rate}
We start with the case \eqref{eq:Fnq}. We emphasize that our analysis is performed in the realm of $\Gamma$-convergence, hence the term $\int_{\Omega^a} H^{a,q}_n(x) \cdot u^a_n(x)dx + \frac{h_n}{r_n^2}\int_{\Omega^b} H^{b,q}_n(x)\cdot u^b_n(x)dx$ with $H^{a,q}_n$ and $H^{b,q}_n$ satisfying \eqref{loadq} 
 can be neglected, i.e. $H$ can be assumed to be $0$ since it is a continuous perturbation, which won't affect our conclusion (see Remark \ref{remloads}). 
In order to provide our asymptotic analysis, we start by providing a compactness result. To this end, we begin with the introduction of the following limiting functions space.

\begin{Definition}[The limit space \(V\)]\label{def:spaceV}
Let $V$ be the space defined as 
      
\begin{equation}
\label{eq:spazioZ}
 V=BV(]0,1[;\RR^3)\x BV(\omega;\RR^3),
\end{equation}
Indeed, with an abuse of notation, we will identify in the sequel fields $u^a \in BV(\Omega^a;\mathbb R^3)$ such that $D_\alpha u^a=\bf 0_\alpha$ with fields, still denoted with the same symbols, in $BV(]0,1[;\mathbb R^3) $ and fields $u^b \in BV(\Omega^b;\mathbb R^3)$ with $D_3 u^b=0$
 with fields $u^b \in BV(\omega;\mathbb R^3)$. \end{Definition}

\begin{Proposition}\label{Compactnessres1}
    Consider the Borel function  \(W : \mathbb{R}^{3\x 3}\to 
    \mathbb{R}\) satisfying
    \eqref{coerci}.  
    For  $n \in \mathbb N$ let \(r_n\), \(h_n\) be such that \eqref{hrzero}, \eqref{ell} and \eqref{loadq} hold. Let the spaces
    \(\mathcal{U}_n^q\) and \(V\) be introduced in \eqref{Vn} and \eqref{eq:spazioZ},
    respectively where the Dirichlet conditions in $\U_n^q$ are as in \eqref{dirichletq}.
    Let $F_n^q$  be the energy functional defined in \eqref{eq:Fnq},
    where the loads $H_n^{a,q}, H_n^{b,q}$ satisfy condition \eqref{loadq}.
    
    Then, for every  
        $(u_n^a, u_n^b) \in \mathcal{U}_n^q$  such that \(\sup_{n\in\N} |F_n^q(  u_n^a,  u_n^b)|
<+\infty\), there exist an increasing sequence of positive integer
        numbers $(n_i)_{i \in \mathbb N}$,  and $({u}^a, {u}^b)
        \in V$, depending possibly on
        the selected subsequence $(n_i)_{i \in \mathbb N}$, such that
       
    \begin{equation}\label{weakdefinitive}\left\{
                \begin{array}{l}
                          u^a_{n_i}\overset{\ast}{\rightharpoonup}  {u}^a  \hbox{ in
                        }BV(\Omega^a;\RR^3),\\\\  u^b_{n_i}\overset{\ast}{\rightharpoonup
                        } {u}^b \hbox{ in
                        }BV(\Omega^b;\RR^3).\end{array}\right.\end{equation}\medskip

%

\end{Proposition}
\begin{proof}[Proof]
 We prove  \eqref{weakdefinitive} in the simplified case in which the loads $H^{a,q}_n$ and $H^{b,q}_n$ are 0, referring to Remark \ref{remloads} for more precise setting in the general case.
        Furthermore we prove only  the first
        convergence, the proof of the second one being similar, and relying on \eqref{eq:Fnq}.

Assumption \eqref{coerci}, 
and the fact that the subset of $\mathcal U^q_n$ of fields with finite energy is not empty, ensures that
        \begin{equation}\label{boundedenergy}
                        \displaystyle{\sup_{n\in \mathbb N}}\Bigg\{\displaystyle{\int_{\Omega^a}\left(\frac{1}{{r_n}}|\nabla_{\alpha}  u^a_n(x)|+
                                |\nabla_3  u^a_n(x)|
                                \right)dx  }
                        + \displaystyle{\frac{h_n}{r^{2}_n}\int_{\Omega^b}\left(|\nabla_{\alpha} 
                                u^b_n(x)|+ \frac{1}{{h_n}}|\nabla_{3}  u^b_n(x)|
                                \right)dx\quad
                                \Bigg\}< +\infty.}
        \end{equation}
        This estimate, \eqref{hrzero} and the boundary conditions in \eqref{dirichletq} in the definition of \(\mathcal{U}^q_n\)  combined with Poincar\'e's inequality (see \cite{Ziemer}),  lead us to conclude that
$(u_n^a)_n$ is uniformly bounded in \(W^{1,1}(\Omega^a;\RR^3)\)
and 
$(u_n^b)_n$ is uniformly bounded in \(W^{1,1}(\Omega^b;\RR^3)\),
with \(\Vert \nabla_{\alpha}
 u^a_n\Vert_{L^1} \leq C r_n \), \(\Vert \nabla_{3}  u^a_n\Vert_{L^1} \leq C, \)
 \(\Vert \nabla_{\alpha}
 u^b_n\Vert_{L^1} \leq \tilde C\), 
 and \(\Vert \nabla_{3}
 u^b_n\Vert_{L^1} \leq \tilde C h_n \) for some positive constants
 \(C\) and \(\tilde C\) that are independent of \(n\).

 Thus, using the compactness properties in \(BV\) of bounded sequences in \(W^{1,1}\) (see Proposition \ref{compBV}), the  compact embedding of \(W^{1,1}\) in \(L^{1^*}\)
(see \cite{AFP})
we can find an increasing sequence of
        positive integer numbers $(n_i)_{i \in \mathbb N}$ and 
        functions $ {u}^a \in BV(\Omega^a;\RR^3)$ and  $ {u}^b \in BV(\Omega^b;\RR^3)$     for which \eqref{weakdefinitive} holds and, in addition,  $ u^a_{n_i}$ converges to  $u^a$ in $L^{1^*}(\Omega^a;\mathbb R^3)$ and weakly* in $BV(\Omega^a;\mathbb R^3)$,  $u^b_{n_i}$ converges to  $u^b$ in $L^{1^*}(\Omega^b;\mathbb
R^3)$, and weakly* in $BV(\Omega^b;\mathbb R^3)$ as $i \to \infty$.
\begin{equation*}
\begin{aligned}
&D_{\alpha} {u}^a =({\bf 0},{\bf 0}) \text{
a.e.~in } \Omega^a, \\
&D_{3}
 u^b = {\bf 0} \text{
a.e.~in } \Omega^b.
\end{aligned}
\end{equation*}
In particular,  $ u^a$ is independent of
        $x_\alpha$ and  $ {u}^b$ is independent of
        $x_3$, which allows us to identify \((u^a, u^b)\) with an element in $V$.
  \end{proof}

\begin{Proposition}[Lower bound $q$]
    \label{lb}
    Consider the Borel function  \(W : \mathbb{R}^{3\x 3}\to 
    \mathbb{R}\) satisfying  \eqref{coerci}.  
    For  $n \in \mathbb N$ let \(r_n\), \(h_n\) be such that \eqref{hrzero} and \eqref{ell}, with $\ell =q \in (0,+\infty)$. Let 
    \(\mathcal{U}_n^q\) and \(V\) be the spaces introduced in \eqref{Vn} and \eqref{eq:spazioZ},
    respectively where the Dirichlet conditions in $\U_n^q$ are as in \eqref{dirichletq}.
    Let $F_n^q$  be the energy functional defined in \eqref{eq:Fnq},
    with the loads $H_n^{a,q},H_n^{b,q}$ which satisfy \eqref{loadq}.

    Assume also that the recession functions of $\hat W$ and of $W_0$ in \eqref{hatW} and \eqref{W0}, respectively, defined as in \eqref{S101rec}, satisfy \eqref{Whatrec} and \eqref{W0rec}.
     
    Then for every  $(u^a, u^b) \in V$, 
    \begin{align*}
        \inf\Big\{&\liminf_n F_n^q(u_n^a, u_n^b): (u_n^a, u_n^b)\in \mathcal U_n^q, 
        \\
        &u_n^a \weakstar u^a \hbox{ in }BV(\Omega^a;\mathbb R^3),
        u_n^b  \weakstar u^b \hbox{ in }BV(\Omega^b;\mathbb R^3)\Big\}\geq
        K^{a,q}({u}^a)+ q
        K^{b,q}({u}^b),
    \end{align*}
    where $K^{a,q}$ and $K^{b,q}$ are defined in \eqref{def:Ka} and \eqref{def:Kb}.
\end{Proposition}
\begin{proof}
    [Proof.] As in the proof of Proposition \ref{Compactnessres1}, we assume for the sake of exposition, that the loads are null.

Considering the convergences in \eqref{weakdefinitive},  take a sequence $(u^a_n, u^b_n)\in \mathcal U^q_n$ for which
\begin{align*}
\liminf_{n\to \infty}\left(\int_{{\Omega}^a}
 W\left(\displaystyle{\frac{1}{r_n}\nabla_{\alpha}u^{a}_n(x)} \;\; | \;\;
\displaystyle{\nabla_{3}u^{a}_n(x)}\right)dx +\frac{h_n}{r_n^2}\int_{{\Omega}^b}
W\left(\displaystyle{\nabla_{\alpha}u^b_n(x)} \;\; | \; \;
\displaystyle{\frac{1}{h_n}\nabla_{3}u^b_n(x)}
\right)dx\right)
<+\infty.
\end{align*}
We want to provide a lower bound.
Then it is easily seen that

\begin{align}
\nonumber
\liminf_{n\to \infty}\left(\int_{{\Omega}^a}
 W\left(\displaystyle{\frac{1}{r_n}\nabla_{\alpha}u^{a}_n(x)} \;\; | \;\;
\displaystyle{\nabla_{3}u^{a}_n(x)}\right)dx +\frac{h_n}{r_n^2}\int_{{\Omega}^b}
W\left(\displaystyle{\nabla_{\alpha}u^b_n(x)} \;\; | \; \;
\displaystyle{\frac{1}{h_n}\nabla_{3}u^b_n(x)}
\right)dx\right) 
\geq \\
\nonumber
\liminf_{n\to \infty}\left(\int_{{\Omega}^a}
 \hat W\left(
\displaystyle{\nabla_{3}u^{a}_n(x)}\right)dx +\frac{h_n}{r_n^2}\int_{{\Omega}^b}
W_0\left(\displaystyle{\nabla_{\alpha}u^b_n(x)}
\right)dx\right)
\geq \\
\nonumber
\liminf_{n\to \infty}\left(\int_{{\Omega}^a}
 \hat W^{\ast \ast}\left(
\displaystyle{\nabla_{3}u^{a}_n(x)}\right)dx +\frac{h_n}{r_n^2}\int_{{\Omega}^b}
Q W_0\left(\displaystyle{\nabla_{\alpha}u^b_n(x)}
\right)dx\right)
\geq \\
\label{lineremark}
 \liminf_{n\to \infty}\int_{{\Omega}^a}
 \hat W^{**}\left(
\displaystyle{\nabla_{3}u^{a}_n(x)}\right)dx + q \liminf_{n\to \infty}\int_{{\Omega}^b}
QW_0\left(\displaystyle{\nabla_{\alpha}u^b_n(x)}
\right)dx,
\end{align}

\noindent with $\hat W$ and $W_0$ as in \eqref{hatW} and \eqref{W0}.
We observe also that in the last but one inequality, we have exploited the following estimates $\hat W \geq \hat W^{**}$ and $W_0 \geq Q W_0$ and in the latter, the superadditivity of $\liminf$.

Moreover, following \cite[page 556]{LeDR1}, the function $QW_0$ can be extended into a function $Z: \mathbb R^3 \times \mathbb R^3$  which is quasiconvex. We recall that, by \cite[Proposition 2.6]{BFMTraces}, $(QW_0)^\infty= Q(W_0^\infty)$ and $((\hat W)^{\ast \ast})^\infty= ((\hat W)^\infty)^{\ast \ast}$, and the analogous properties of \eqref{Whatrec} and \eqref{W0rec} hold. Notice that $u_n^a|_{\omega\x\{1\}}=c^{a,q}$ and $u_n^b|_{\partial\omega\x]-1,0[}=f^{b,q}$.  We apply now
\cite[Example 4.1]{BFMTraces} in both terms to the functionals $W^{1,1}(\Omega^a;\mathbb R^3)  \ni u^a \to E(u^a):= \int_{\Omega^a} \hat W^{\ast \ast}(\nabla_3 u^a)dx + \varepsilon \|\nabla_\alpha u^a\|_{L^1(\Omega^a}$ and $W^{1,1}(\Omega^a;\mathbb R^3)  \ni u^b \to E(u^b):= \int_{\Omega^b} QW_0(\nabla_\alpha u^b) dx+ \|\nabla_3 u^b\|_{L^1(\Omega^b)} $to obtain 
\begin{align*}
&\liminf_{n\to \infty}\int_{{\Omega}^a}
 \hat W^{**}\left(
\displaystyle{\nabla_{3}u^{a}_n(x)}\right)dx + q \liminf_{n\to \infty}\int_{{\Omega}^b}
QW_0\left(\displaystyle{\nabla_{\alpha}u^b_n(x)}
\right)dx \geq 
\\
\geq&\;\int_{{\Omega}^a}
 \hat W^{**}\left(
\displaystyle{\nabla_{3}u^{a}(x_3)}\right)dx + \int_{\Omega^a} (\hat W^{**})^\infty \left(\frac{d D^s_\alpha u^a}{d |D^s_\alpha u^a|}\right) d |D^s_\alpha u^a | +(\hat W^{\ast \ast})^\infty(c^{a,q}- u^a(1^-))+ 
\\
+\;&q \left(\int_{{\Omega}^b}
QW_0\left(\displaystyle{\nabla_{\alpha}u^b(x_\alpha)}
\right)dx +\int_{{\Omega}^b}
(QW_0)^\infty\left(\displaystyle{\frac{d D^s_{\alpha}u^b}{ d |D^s_\alpha u^b|}}\right)d |D^s_\alpha u^b| + \right. 
\\
+\;&\left.\int_{\partial \omega \times ]-1,0[}(Q W_0)^\infty ((f^{b,q}- (u^b)^-)\otimes \nu_{\partial \omega \times ]-1,0[}) d \mathcal H^2 \right) = K^{a,q}(u^a)+ q K^{b,q}(u^b).
\end{align*}
 This concludes the proof, observing that the $\nu_{\partial \omega \times ]-1,0[}$ represents the normal to the discontinuity set of $u^b$.
\end{proof}

Let \begin{equation}
    \label{Udense}
        U:=\left \{(u^a, u^b)
\in(c^{a,q} + C^\infty_a([0,1];\mathbb R^3))\times (f^{b,q} + C^\infty_0(\omega;\mathbb R^3)), 
u^a(0) = u^b(0_\alpha)\right\} 
\end{equation}
with $C^\infty_a([0, 1];\mathbb R^3) =
\{u^a \in C^\infty([0, 1];\mathbb R^3) : u^a(1) = 0\}.$

\begin{Proposition}
        \label{density1storder}
        Let $U$ be as in \eqref{Udense}. Then it is dense in
        $\left(c^{a,q}+W^{1,1}_a(]0,1[;\mathbb R^3)\right) \x$ \\ $  
      \left(f^{b,q}+W^{1,1}_0(\omega;\mathbb R^3)\right) $
with respect to the strong topology of  $W^{1,1}(]0,1[;\mathbb R^3)\times W^{1,1}(\omega;\mathbb R^3)$.
\end{Proposition}
\begin{proof}
        It follows as the case $p < N-1$ in Proposition 3.1 of \cite{GGLM1}, with $v^{(1)}$ and $v^{(2)}$ therein, coinciding with $u^a$ and $u^b$, respectively. Indeed, we rely  on equations (3.2), (3.3) and (3.4) in Proposition 3.1 of \cite{GGLM1}, considering a sequence of cut-off functions  $\varphi_n$ as in (3.3) of class $C^\infty$ with the same radii as in (3.2), so that (3.4) holds. Hence the sequences $$u^a_n(x_3):= u^a(x_3) \hbox{ for  }x_3 \in ]0,1[,$$ and $$u^b_n(x_\alpha):=\left\{\begin{array}{ll}
        u^a(0) &\hbox{ in } \overline B(0_\alpha, \varepsilon_n),\\
        \varphi_n u^a(0)+ (1-\varphi_n(x_\alpha))u^b(x_\alpha) & \hbox{ in }\overline B(0_\alpha, \eta_n)\setminus \overline B(0_\alpha, \varepsilon_n),\\
        u^b(x_\alpha) &\hbox{ in }\omega \setminus \overline B(0_\alpha, \eta_n)
        \end{array}\right. $$ for $x_\alpha \in \omega$, converge, as in the assertion, to $u^a$ and $u^b$, respectively.

\end{proof}

\begin{Proposition}\label{ub1stestimateU}
Under the same assumptions of Proposition \ref{lb}, for every $(u^a, u^b) \in U$ in \eqref{Udense},
\begin{align}
\inf\Big\{&\limsup_n F_n^q(u_n^a,u_n^b): (u_n^a, u_n^b)\in \mathcal U_n^q, \; u_n^a \overset{\ast}{\rightharpoonup} u^a \hbox{ in }BV(\Omega^a;\mathbb R^3), \nonumber\\
& u_n^b  \overset{\ast}{\rightharpoonup} u^b \hbox{ in }BV(\Omega^b;\mathbb R^3)\Big\}\leq
\int_0^1 \hat W(\nabla_3{u}^a(x_3))dx_3 + q
\int_\omega W_0(\nabla_\alpha {u}^b(x_\alpha))dx_\alpha+ \nonumber\\
&\int_{\Omega^a} H(x) \cdot u^a(x_3) dx + \int_{\Omega^b }H (x)\cdot u^b(x_\alpha) dx.
\label{ub1st}
\end{align}
\end{Proposition}
\begin{proof}[Proof]  As in Proposition \ref{lb}
 we assume that the loads are null, so we can avoid to consider them in the analysis below.
 
Following the strategy in \cite[Section 5]{GZNODEA}, for every $(u^a, u^b)\in U$ and $z^a \in W^{1,1}_0(]0,1[;\mathbb R^3)$ and $z^b \in W^{1,1}_0(\omega;\mathbb R^3)$ we define
\begin{align*}
u_n^a(x_\alpha, x_3):=\left\{ \begin{array}{ll} (z^a (\varepsilon_n )r_n x_\alpha
+ u^a(\varepsilon_n)) \frac{x_3}{\varepsilon_n} + u^b(r_n x_\alpha) \frac{\varepsilon_n-x_3}{\varepsilon_n} &\hbox{ if } (x_\alpha, x_3)\in \omega \times ]0,\varepsilon_n[,\\
z^a(x_3) r_n x_\alpha + u^a(x_3) &\hbox{ if }(x_\alpha, x_3) \in \omega \times [\varepsilon_n, 1[,
\end{array}\right.
\end{align*}
\begin{align*}
u^b_n(x_\alpha, x_3):= h_n x_3 z^b(x_\alpha)+ u^b(x_\alpha) \hbox{ if } (x_\alpha, x_3) \in \Omega^b,
\end{align*}

for a suitable sequence $\varepsilon \to 0^+$ (cf. \cite{GZNODEA}).
Arguing as in \cite[(4.11), (4.12), (4.5) in Proposition 4.1]{GGLM1},  for $p=1$, it results that 
\begin{align}
u^a_n \to u^a &\hbox{ strongly in }L^1(\Omega^a;\mathbb R^3)\nonumber
\\
\frac{1}{r_n}\nabla_\alpha u^a_n \to z^a &\hbox{ strongly in } L^1(\Omega^a; \mathbb R^{3\x2}), \label{convubaquote}\\
\nabla_3 u^a_n \to \nabla_3 u^a &\hbox{ strongly in }L^{1}(\Omega^a;\mathbb R^3).\nonumber
\end{align}
Moreover, as in \cite{LeDR1},
\begin{align}
u^b_n \to u^b &\hbox{ strongly in }L^1(\Omega^b;\mathbb R^3),\nonumber\\
\nabla_\alpha u^b_n \to \nabla_\alpha u^b &\hbox{ strongly in } L^1(\Omega^b; \mathbb R^{3\x 2}),\label{convubbquote}\\
\frac{1}{h_n} \nabla_3 u^b_n = z^b &\hbox{ in }\Omega^b, \nonumber
\end{align}
\begin{align}\label{limenquote}
\lim_{n\to \infty} \left(K^{a,q}_n(u^a_n)+\frac{h_n}{r_n^2} K^{b,q}_n(u^b_n)\right) =\int_{\Omega^a} W( z^a(x_3), \nabla_3 u^a(x_3))dx + q \int_{\Omega^b} W(\nabla_\alpha u^b(x_\alpha), z^b(x_\alpha)) dx.
\end{align}

Recalling that  $z^a \in  W^{1,1}_0(]0,1[;\mathbb R^3)$ and $z^b \in W^{1,1}_0(\omega;\mathbb R^3)$ can be arbitrarily chosen, 
then arguing as in the proof of \cite[Theorem 5.1]{GZNODEA}, replacing the spaces $W^{1,p}_0(]0,1[;\mathbb R^3)$ and $W^{1,p}_0(\omega;\mathbb R^3)$ by $W^{1,1}_0(]0,1[;\mathbb R^3)$ and $ W^{1,1}_0(\omega;\mathbb R^3)$ and the spaces $L^p(]0,1[;\mathbb R^3)$ and $L^p(\omega;\mathbb R^3)$ by $L^1(]0,1[;\mathbb R^3)$ and $L^1(\omega;\mathbb R^3)$, respectively, using Aumann'selection lemma, and definitions $W_0$ and $\hat W$ in \eqref{hatW} and \eqref{W0} respectively, 
we obtain
\begin{align}
\inf&\left\{\liminf_n F_n^q(u_n^a, u^b_n): (u_n^a, u^b_n) \in \mathcal U_n^q,  u_n^a \overset{\ast}{\rightharpoonup} u^a \hbox{ in }BV(\Omega^a;\mathbb R^3 ),
 \right. \nonumber\\
&\left. u_n^b  \overset{\ast}{\rightharpoonup} u^b \hbox{ in }BV(\Omega^b;\mathbb R^3 ) 
\right\} \leq \nonumber
\int_0^1 \hat W(\nabla_3 u^a(x_3)) dx_3+ q\int_\omega W_0(\nabla_{\alpha}{u}^b(x_\alpha))dx_\alpha.
\end{align}

The proof is concluded by a standard diagonalization result applying Proposition \ref{density1storder}.

\end{proof}


\begin{Proposition}[Upper bound $q$]\label{2ndestimateub}
Under the same assumptions of Proposition \ref{lb}, for every $(u^a, u^b) \in V$
\begin{align}
\inf\left\{\limsup_n F_n^q(u_n^a, u_n^b): (u_n^a, u_n^b)\in \mathcal U_n^q, u_n^a \overset{\ast}{\rightharpoonup} u^a \hbox{ in }BV(\Omega^a;\mathbb R^3),  \right. \nonumber \\
\left.
u_n^b  \overset{\ast}{\rightharpoonup} u^b \hbox{ in }BV(\Omega^b;\mathbb R^3)\right\}\leq K^{a,q}(u^a)+ q K^{b,q}(u^b),
\label{ub}
\end{align}
where $K^{a,q}$ and $K^{b,q}$ are defined in \eqref{def:Ka} and \eqref{def:Kb}, respectively.
\end{Proposition}

\begin{proof}[Proof]
In view of Propositions \ref{density1storder} and \ref{ub1stestimateU},  for every $u^a\in c^{a,q}+W^{1,1}_a(]0,1[;\mathbb R^3)$ 
and $u^b \in f^{b,q}+W^{1,1}_0(\omega;\mathbb R^3)$, inequality \eqref{ub1st} holds.
     
Recall that $\mathcal R^q$ in \eqref{def:Rell} coincides with $V$ in Definition \ref{def:spaceV} and let the functional $G^q: V\to \mathbb R$ 
defined in  \eqref{normal_inf}, be given by 
\begin{align*}
G^q(u^a, u^b):=\inf&\left\{\liminf_n F_n^q(u_n^a,u^b_n): (u_n^a, u_n^b)\in \mathcal U_n^q, \; u_n^a \overset{\ast}{\rightharpoonup} u^a \hbox{ in }BV(\Omega^a;\mathbb R^3),  
u_n^b  \overset{\ast}{\rightharpoonup} u^b \hbox{ in }BV(\Omega^b;\mathbb R^3)\right\}.
\end{align*} 

Its lower semicontinuity with respect to the weak * topology of $BV(]0,1[;\mathbb R^3) \times BV(\omega;\mathbb R^3)$ (see Lemma \ref{preFlsc}) and the fact that the functional in the right hand side of \eqref{ub1st},  i.e. the functional  $\int_0^1 \hat W(\nabla_3 u^a(x_3)) dx_3+ q\int_\omega W_0(\nabla_{\alpha}{u}^b(x_\alpha))dx_\alpha$   defined in $c^{a,q}+ W^{1,1}_a(]0,1[;\mathbb R^3)\times f^{b,q}+W^{1,1}_0(\omega;\mathbb R^3)$, admits, in view of \cite[Example 4.1]{BFMTraces} as relaxation in $V$, (with respect to the $BV$ weak * convergence in $BV(]0,1[;\mathbb R^3)\times BV(\omega;\mathbb R^3)$) 
the same functional, with $K^{a,q}$ and $K^{b,q}$ given by \eqref{def:Ka} and \eqref{def:Kb}, respectively,
proves our result.
\end{proof}

\begin{Remark}
    \label{general_dirichlet1}
    Propositions \eqref{lb}, \eqref{ub1stestimateU} and \eqref{2ndestimateub} also hold if in \eqref{dirichletq} we consider more general boundary conditions, i.e.  any $d^{a,q}\in\RI^{3\x2}$ and $g^{b,q}\in W^{1,1}(\omega;\RI^3)$.

   In fact, suppose that in the proof of Proposition \ref{lb}, in \eqref{lineremark} we have $u_n^a|_{\o\x\{1\}}=c^{a,q}+r_nd^{a,q}x_\alpha$ for $d^{a,q}\in\RI^{3\x2}$. Then for $\tilde u_n^a := u_n^a - r_nd^{a,q}x_\alpha$, since $d^{a,q}$ is constant, we have that $\nabla_3\tilde u_n^a =\nabla_3 u_n^a $ and $\tilde u_n^a|_{\omega\x\{1\}}=c^{a,q}$ therefore 
    $$
     \liminf_{n\to \infty}\int_{{\Omega}^a}
 \hat W^{**}\left(
\displaystyle{\nabla_{3}u^{a}_n}(x)\right)dx = \liminf_{n\to \infty}\int_{{\Omega}^a}
 \hat W^{**}\left(
\displaystyle{\nabla_{3}\tilde u^{a}_n}(x)\right)dx.
    $$

    \vspace{4mm}
    For $u_n^b \in W^{1,1}(\Omega^b;\RI^3)$ with $u_n^b|_{\partial\omega\x]-1,0[}=f^{b,q}+h_ng^{b,q}x_3$ consider $\tilde u_n^b := u_n^b - h_ng^{b,q}x_3$.
    \begin{align*}
        \liminf_{n\to \infty}\int_{{\Omega}^b}
        QW_0\left(\displaystyle{\nabla_{\alpha}\tilde u^b_n(x)}
        \right)dx = \liminf_{n\to \infty}\int_{{\Omega}^b}
        QW_0\left(\displaystyle{\nabla_{\alpha} u^b_n(x) - h_nx_3\nabla_\alpha g^{b,q}(x_\alpha)}
        \right)dx.
    \end{align*}
    Since $QW_0$ is quasiconvex and has linear growth it is Lipschitz (see \cite{DA}), hence
   
    \begin{align*}
        \liminf_{n\to \infty}\int_{{\Omega}^b}
        QW_0\left(\displaystyle{\nabla_{\alpha}\tilde u^b_n(x)}
        \right)dx = \liminf_{n\to \infty}\int_{{\Omega}^b}
        QW_0\left(\displaystyle{\nabla_{\alpha} u^b_n(x)}
        \right)dx.
    \end{align*}
   Therefore, we can reproduce the rest of the proof in Proposition \ref{lb} with $\tilde u_n^a$ and $\tilde u_n^b$.

    Finally, in Proposition \ref{ub1stestimateU} by replacing $z^a$ and $z^b$ respectively with $z^a+d^{a,q}$ and $z^b+g^{b,q}$ we still reach the same conclusion.
    
\end{Remark}

\subsection{Infinite rate: $\ell=+\infty$.}\label{sec:rateinfty}

We now consider the case \eqref{eq:Fninfty}. For the reasons exposed in Remark \ref{remloads}, we can neglect the term $\int_{\Omega^a} H_n^{a,\infty}(x) \cdot u^a_n(x)dx + \frac{h_n}{r_n^2}\int_{\Omega^b} H_n^{b,\infty}(x)\cdot u^b_n(x)dx$ since \eqref{loadinfty} provides continuity of these latter integrals with respect to $\Gamma$-convergence.

\begin{Proposition}\label{Compactness_infty}
    Consider the Borel function  \(W : \mathbb{R}^{3\x 3}\to 
    \mathbb{R}\) satisfying \eqref{coerci} and \eqref{eq:growth1}.  
    For  $n \in \mathbb N$ let \(r_n\), \(h_n\) be as in \eqref{hrzero}, with $\ell$ in \eqref{ell} coinciding with $+\infty.$ Let the spaces
    \(\mathcal{U}_n^\infty\) be as in \eqref{Vn}
    with Dirichlet conditions as in \eqref{dirichletinfty}, and the loads $H_n^a,H_n^b$ satisfy \eqref{loadinfty}.
    Let $F_n^\infty$  be the energy functional defined in \eqref{eq:Fninfty}.

 Then, for every  
        $( u_n^a, u_n^b) \in \mathcal{U}^\infty_n$  such that \(\sup_{n\in\N} |F_n^\infty(  u_n^a,  u_n^b)|
<+\infty\), there exist an increasing sequence of positive integer
        numbers $(n_i)_{i \in \mathbb N}$,
        ${u}^a
        \in BV(]0,1[;\RI^3)$
        depending possibly on
        the selected subsequence $(n_i)_{i \in \mathbb N}$, such that
       
        \begin{equation}\label{convinftycomp}\left\{
                \begin{array}{l}
                          u^a_{n_i}\overset{\ast}{\rightharpoonup}  {u}^a  \hbox{ in
                        }BV(\Omega^a;\RR^3),\\\\  u^b_{n_i} \to (x_\alpha,0) \hbox{ in
                        }W^{1,1}(\Omega^b;\RR^3).\end{array}\right.\end{equation}
        
         %
\noindent Furthermore, it results
         \begin{equation*}
         \displaystyle{\frac{1}{h_{n_i}}}\nabla_{3}  u^b_{n_i}
                        \to  (0_\alpha,1) \hbox{  in
                        }L^1(\Omega^b;\mathbb R^3).
        \end{equation*}       
\end{Proposition}

\begin{proof}

    Arguing as in Proposition \ref{Compactnessres1}, we first observe that we neglect to loads, then we can say that there exist a not relabelled subsequence $(u^a_{n}, u^b_{n})\in \mathcal U_n^\infty$,
    and $(u^a,u^b)\in V$ (where V is defined in \eqref{eq:spazioZ}) such that \eqref{weakdefinitive} holds. 
    From \eqref{eq:growth1} and
     the assumption $\sup_{n\in\N} |F_n^\infty(  u_n^a,  u_n^b)|
<+\infty$, there exists a $C>0$, such that,  for all $n\in\N$
    $$
  \frac{h_n}{r_n^2} \int_{{\Omega}^b}
\left|\;\left(\displaystyle{\nabla_{\alpha}u^b_n(x)} \;\; | \; \;
\displaystyle{\frac{1}{h_n}\nabla_{3}u^b_n(x)}
\right) \; \; - \; Id \right| dx \leq  \frac{h_n}{r_n^2}K_n^{b,\infty}(u_n^b) \leq C
    $$
and so,

\begin{align}
    \label{growth_alpha_infty}
    \int_{{\Omega}^b}
\left|\nabla_{\alpha}u^b_n(x) 
 \;  - \;Id_\alpha \right| dx \leq \frac{r_n^2}{h_n}C
\\
\label{growth_3_infty}
\int_{{\Omega}^b}
\left|\frac{1}{h_n}\nabla_{3}u^b_n(x) 
 \;  - \;(0_\alpha,1)\right| dx \leq \frac{r_n^2}{h_n}C,
\end{align}
with $Id_\alpha$ 
as in \eqref{matrixdef},
which from \eqref{eq:Fninfty} implies

\begin{align}
    \nabla u_n^b \to \left(Id_\alpha \;\; | \; \; {\bf 0} \right) \; \hbox{in} \; L^1(\Omega^b;\RI^{3\times 3}) \label{58}
    \\
     \displaystyle{\frac{1}{h_{n}}}\nabla_{3}  u^b_{n}
                        \to  (0_\alpha, 1) \hbox{  in
                        } L^1(\Omega^b;\mathbb R^3),\nonumber
\end{align}
where ${\bf 0}$ in the \eqref{58} stands for the null column in $\mathbb R^{3\times 1}.$

From \eqref{weakdefinitive} and \eqref{58} we get $ u^b_{n_i} \to (x_\alpha,0) + y_0\; \hbox{ in} \; W^{1,1}(\Omega^b;\RI^3)$ for some $y_0\in \RI^3$. Since $u_n^b|_{\partial\omega\x]-1,0[}=(x_\alpha,h_nx_3)$, from the strong $W^{1,1}$ convergence we conclude that the traces also converge strongly in $L^1$ to $(x_\alpha,0) + y_0$ and therefore $y_0={\bf 0}$.

\end{proof}

In conclusion, in view of Proposition \ref{Compactness_infty},
the functional $G^\infty$ in \eqref{normal_inf} becomes the left hand side of \eqref{probl_infty,infty}, i.e.

\begin{align}\label{probl_infty}
G^\infty(u^a,(x_\alpha, 0)):=\inf\left\{\liminf_{n\to\infty}F_n^\infty(u_n^a, u_n^b): (u_n^a, u_n^b)\in \mathcal U_n^\infty,  \right. \nonumber\\
\left.u_n^a \overset{\ast}{\rightharpoonup} u^a \hbox{ in }BV(\Omega^a;\mathbb R^3), 
u_n^b  \to (x_\alpha,0) \hbox{ in }W^{1,1}(\Omega^b;\mathbb R^3)\right\},
\end{align}

where $\mathcal U_n^\infty$ is given by \eqref{Vn} with the specific boundary datum as in \eqref{dirichletinfty} and $u^a \in BV(]0,1[;\mathbb R^3)$.

The functional $G^\infty$ is defined in $\mathcal R^\infty$, namely 
in $BV(]0,1[;\mathbb R^3) \times \{(x_\alpha,0)\}$.

\noindent We will study problem \eqref{probl_infty} in two separate subcases: first $\frac{h_n}{r_n} \to 0$, then $\frac{h_n}{r_n}\to \infty$. 
In what follows, this subcases will be denoted respectively by $(\infty,0)$ and $(\infty, \infty)$.

\subsubsection{Case $\displaystyle\lim_{n\to \infty}\frac{h_n}{r_n}= 0$}

As in the previous cases, the integral representation will be achieved in two steps.

\begin{Proposition}[Lower bound ($\infty, 0$)]
    \label{lbinfty0}
  Consider the Borel function  \(W : \mathbb{R}^{3\x 3}\to 
    \mathbb{R}\) satisfying  \eqref{coerci} and \eqref{eq:growth1}.   Assume also that the recession function of $\hat W$ \eqref{hatW} defined as in \eqref{S101rec}, satisfies \eqref{Whatrec}. 
    For  $n \in \mathbb N$ let \(r_n\), \(h_n\) be as in \eqref{hrzero}, with $\ell$ in \eqref{ell} coinciding with $+\infty,$ with $\lim_{n\to \infty}\frac{h_n}{r_n}=0$. Let the spaces
    \(\mathcal{U}_n^\infty\) be as in \eqref{Vn}
    with Dirichlet conditions as in \eqref{dirichletinfty}, and the loads $H_n^{a,\infty},H_n^{b,\infty}$ satisfy \eqref{loadinfty}.
    Let $F_n^\infty$  be the energy functional defined in \eqref{eq:Fninfty},
    and let $G^\infty(u^a, (x_\alpha, 0))$ be as in \eqref{probl_infty,0}.
  
    Then for every  $u^a \in BV(]0,1[;\mathbb R^3)$ 
    \begin{align*}
    G^\infty(u^a, (x_\alpha, 0)) \geq K^{a,\infty}(u^a).
\end{align*}
\end{Proposition}
\begin{proof}
    [Proof.]
    Recall that we consider null loads. Let $K^{a,\infty}_n$ be the functional in \eqref{kna}.
The energy \begin{align}\label{probl_string}
\inf\left\{\lim_{n\to \infty}\left(K^{a,\infty}_n(  u^a_n)\right): u_n^a \in \left(c^{a,\infty} +W_a^{1,1}
(\Omega^a;\RR^3)\right),  
u_n^a \overset{\ast}{\rightharpoonup} u^a \hbox{ in }BV(\Omega^a;\mathbb R^3) \right\}
\end{align}
 is clearly a lower bound to \eqref{probl_infty}. Using the same arguments as in Proposition \ref{lb} we can get $K^{a,\infty}(u^a)$ as a lower bound to \eqref{probl_string} and consequently to \eqref{probl_infty}, which concludes the proof.
\end{proof}

In order to prove the upper bound we assume first $u^a\in (c^{a,\infty} + W_a^{1,1}(\;]0,1[;\RI^3))$, where $c^{a,\infty}$ is as in \eqref{dirichletinfty}, and we  get the following result.

\begin{Lemma}\label{ub_sobolev_infty}
Under the same assumptions of Proposition \ref{lbinfty0}, it results that
\begin{align*}
\inf\left\{\liminf_nF_n^\infty(u_n^a, u^b_n): (u_n^a, u_n^b)\in \mathcal U_n^\infty,  
u_n^a \overset{\ast}{\rightharpoonup} u^a \hbox{ in }BV(\Omega^a;\mathbb R^3), 
\right.
\nonumber\\
\left.
u_n^b  \to (x_\alpha,0) \hbox{ in }W^{1,1}(\Omega^b;\mathbb R^3)\right\} \; \leq \;
\int_0^1 \hat{W}(\nabla_3{u}^a(x_3))dx_3,
\end{align*}
for every $u^a \in (c^{a,\infty} + W_a^{1,1}(\;]0,1[;\RI^3))$. 
\end{Lemma}

\begin{proof}
    For $z^a\in W^{1,1}_0(]0,1[;\RI^{3\x2})$ consider the following sequences:

\begin{align}
\label{sequence_un_above}
    &u_n^a(x_\alpha,x_3) = z^a(x_3) r_nx_\alpha + u^a(x_3) & (x_\alpha,x_3)\in \Omega^a,\\
    &u_n^b(x_\alpha,x_3) = (x_\alpha,h_nx_3)(1-\phi_n(x_\alpha)) \; + \; u_n^a\left(\frac{1}{r_n}x_\alpha, -h_nx_3\right)\phi_n(x_\alpha) \quad &(x_\alpha,x_3)\in\Omega^b, \label{seq_un_below}
\end{align}
where $\phi_n$ is a smooth cut-off function defined in $\RI^2$ satisfying:

\begin{align}
&\phi_n(x_\alpha):= \left\{\begin{array}{ll}
        1 &\hbox{ if } x_\alpha\in r_n\omega,\\
        0 &\hbox{ if } x_\alpha \notin (1+\beta)r_n\omega, 
        \end{array}\right. \nonumber \\ \label{cutoff}
&|\nabla\phi_n(x_\alpha)| \leq \frac{C}{\beta r_n} \quad \hbox{and} \quad 0 \leq \phi_n(x_\alpha) \leq 1.
\end{align}
Here $\beta > 0$ is a fixed arbitrary value.

Notice that $(u_n^a,u_n^b)\in \mathcal{U}_n^\infty$ because for $x_\alpha\in r_n\omega$ we get $u_n^b(x_\alpha,0) = u_n^a(x_\alpha /r_n, 0)$.
\begin{align*}
    \left(\frac{1}{r_n}\nabla_\alpha u_n^a \; | \; \nabla_3 u_n^a \right) =& \; ( z^a  
    \; | \;  \nabla_3 u^a + \nabla_3z^a 
    r_nx_\alpha  ) 
    \\ 
    \\ 
    \left(\nabla_\alpha u_n^b \; | \; \frac{1}{h_n} \nabla_3 u_n^b\right)(x_\alpha,x_3) =&  Id(1-\phi_n(x_\alpha)) 
    + \left( \frac{1}{r_n} \nabla_\alpha u_n^a \; | \; -\nabla_3 u_n^a \right) \left( \frac{x_\alpha}{r_n},-h_nx_3\right)\cdot\phi_n(x_\alpha) 
    \\
    &+\left(   \nabla_\alpha\phi_n \otimes \left[ u_n^a\left(\frac{x_\alpha}{r_n},-h_n x_3\right)-(x_\alpha,h_nx_3) \right] \; | \;  0\right).
\end{align*}

Arguing as in Proposition \ref{ub1stestimateU} it results that
\begin{align}
u^a_n \to u^a &\hbox{ strongly in }L^1(\Omega^a;\mathbb R^3) \nonumber
\\
\label{conv_alpha}
\frac{1}{r_n}\nabla_\alpha u^a_n \to z^a &\hbox{ strongly in } L^1(\Omega^a; \mathbb R^{3\x2}),
\\
\label{conv_3}
\nabla_3 u^a_n \to \nabla_3 u^a &\hbox{ strongly in }L^{1}(\Omega^a;\mathbb R^3),
\end{align}
which implies, since $ W$ is continuous,

$$
K^{a,\infty}_n(u_n^a) = \int_{\Omega^a} W \left(\frac{1}{r_n}\nabla_\alpha u_n^a(x) \; | \; \nabla_3 u_n^a(x) \right) dx \to \int_{\Omega^a} W \left(z^a(x_3) \; | \; \nabla_3 u^a(x_3) \right) dx.
$$

For the sequence $(u_n^b)_n$ in \eqref{seq_un_below}, it is easily seen that $u_n^b|_{\partial\omega\x]-1,0[}= (x_\alpha, h_nx_3)$. To prove that $(u_n^b)_n$ is admissible, we observe that, by the Proposition \ref{Compactness_infty}, if $\sup_{n\in \N} \frac{h_n}{r_n^2} K_n^{b,\infty}(u_n^b) < +\infty$, as it is shown below, then $u_n^b\to(x_\alpha,0)$ in $W^{1,1}(\Omega^b;\RI^3)$ .

From \eqref{coerci}, \eqref{eq:growth1} and \eqref{cutoff} we have the following estimation for the energy, where 
$$B_n:=(1+\beta)r_n\omega\x]-1,0[, \quad A_n:=r_n\omega\x]-1,0[,\quad  \psi_n^b(x_\alpha,x_3) := (x_\alpha,h_nx_3)$$

\begin{align}
    \frac{h_n}{r_n^2}K_n^{b,\infty}&(u_n^b(x)) = \;  \frac{h_n}{r_n^2}\int_{{\Omega}^b} W\left(\displaystyle{\nabla_{\alpha}u^b_n(x)} \;\; | \; \;
\displaystyle{\frac{1}{h_n}\nabla_{3}u^b_n(x)} 
\right)dx &\leq \nonumber 
\\ 
\nonumber
&\frac{h_n}{r_n^2}\left[\int_{{\Omega}^b \setminus B_n}W\left(\displaystyle{\nabla_{\alpha}u^b_n(x)} \; |  \;
\displaystyle{\frac{1}{h_n}\nabla_{3}u^b_n(x)} 
\right)dx + C\int_{B_n} \left[1 + \left| \left(\displaystyle{\nabla_{\alpha}u^b_n(x)} \; |  \;
\displaystyle{\frac{1}{h_n}\nabla_{3}u^b_n(x)} 
\right) \right|\right]dx\right]
&\leq
\\
\nonumber
\\
&\frac{h_n}{r_n^2}C|B_n| + \frac{h_n}{r_n^2}C\int_{B_n \setminus A_n}  (1-\phi_n(x_\alpha))\left| Id \right|dx
&+ \nonumber \\
  &\frac{h_n}{r_n^2}C\int_{B_n \setminus A_n}  \left|\left(u^a_n\left(\frac{x_\alpha}{r_n},-h_n x_3\right) \; - \; \psi_n^b(x) \right) \otimes \left(\displaystyle{\nabla_{\alpha}\phi_n(x_\alpha)} \;\; | \; \;
 0
\right) \right|dx &+
\nonumber 
\\
&\frac{h_n}{r_n^2}C\int_{B_n}  \phi_n(x_\alpha)\cdot\left|\left(\displaystyle{\frac{1}{r_n}\nabla_{\alpha}u^a_n} \;\; | \; \;
\displaystyle{-\nabla_{3}u^a_n}
\right)\left(\frac{x_\alpha}{r_n},-h_n x_3\right) \right|dx &\leq 
\nonumber
\\
\nonumber
\\ 
\label{int1}
&O(h_n) + 
 \frac{h_n}{r_n^2}C\int_{B_n}  \frac{1}{\beta r_n} \left|\left(u^a_n\left(\frac{x_\alpha}{r_n},-h_n x_3\right) \; - \; \psi_n^b(x) \right) \right|dx \ +
\\
\label{int2}
&\frac{h_n}{r_n^2}C\int_{B_n}  \left|\left(\displaystyle{\frac{1}{r_n}\nabla_{\alpha}u^a_n} \;\; | \; \;
\displaystyle{-\nabla_{3}u^a_n}
\right)\left(\frac{x_\alpha}{r_n},-h_n x_3\right) \right|dx.
\end{align}

Since $u^a\in W^{1,1}(]0,1[\;\RI^3)$, by the continuous embedding in $L^\infty$ (see \cite{Brezis}), it follows that $u^a$ is bounded in $L^\infty(]0,1[;\mathbb R^3])$. Hence, from \eqref{sequence_un_above} $(u^a_n)_n$ is also uniformly bounded in $L^\infty(]0,1[;\mathbb R^3)$. Noticing that $(\psi_n^b)_n$ is also bounded in $L^\infty(\Omega^b;\mathbb R^3)$ it follows that the integral in \eqref{int1} is $O(h_n/r_n)$ and therefore it goes to 0 as $n\to\infty$.

On the other hand, \eqref{int2}, after a change of variables,  equals to
$$
C\int_{\omega\x]0,h_n[} \left| \left(\frac{1}{r_n}\nabla_\alpha u_n^a \; | \; \nabla_3 u_n^a \right)\right| dx.
$$
From \eqref{conv_alpha}, \eqref{conv_3} and uniform integrability it results that the integral tends to 0.

Consequently for all $z^a \in W_0^{1,1}(]0,1[;\RI^{3\x2})$ and with $u^a\in (c^{a,\infty} + W_a^{1,1}(]0,1[;\RI^{3}))$, 
we have
\begin{align*}
\inf\left\{\liminf_n\left(K^{a,\infty}_n(  u^a_n)+
\frac{h_n}{r_n^{2}}K^{b,\infty}_n(  u^b_n)\right): (u_n^a, u_n^b)\in \mathcal U_n, u_n^a \overset{\ast}{\rightharpoonup} u^a \hbox{ in }BV(\Omega^a;\mathbb R^3),  \right. \\
\left.
u_n^b  \to (x_\alpha,0) \hbox{ in }W^{1,1}(\Omega^b;\mathbb R^3)\right\} \; \leq \; \int_{\Omega^a} W \left(z^a(x_3) \; | \; \nabla_3 u^a(x_3) \right) dx,
\end{align*}
where we recall that $z^a$ has been arbitrarily chosen.

Following along the lines the arguments in \cite[Proposition~7]{LeDR1}, relying on the density of $W^{1,1}_0(]0,1[;\mathbb R^{3\times 2})$ in $L^1(]0,1[;\mathbb R^{3\times 2})$ with respect to the $L^1$-strong convergence, and exploiting the dominated convergence theorem, and the definition of $\hat W$ in \eqref{hatW}, we have
$$
\inf \left\{  \int_{\Omega^a} W \left(z^a(x_3) \; | \; \nabla_3 u^a(x_3) \right) dx \; : \: z^a \in W_0^{1,1}(\Omega^a;\RI^{3\x2}) \right\} \leq \int_{\Omega^a} \hat{W}(\nabla_3 u^a(x_3))dx,
$$
thus the proof is completed.
\end{proof}

The following result, together with Proposition \ref{lbinfty0}, concludes the proof of Theorem \ref{generalrepresentation} (case $\ell=+\infty$, $\lim_{n\to \infty}\frac{h_n}{r_n}=0$), since it provides an upper bound for the relaxed energy in \eqref{probl_infty}.

\begin{Proposition}[Upper bound $(\infty,0)$]\label{ub_infty,0}
Under the same assumptions of Proposition \ref{lbinfty0}, it results 
\begin{align*}
\inf\left\{\liminf_nF_n^\infty(u^a_n,u^b_n): (u_n^a, u_n^b)\in \mathcal U_n^\infty,  
u_n^a \overset{\ast}{\rightharpoonup} u^a \hbox{ in }BV(\Omega^a;\mathbb R^3), 
\right.
\nonumber\\
\left.
u_n^b  \to (x_\alpha,0) \hbox{ in }W^{1,1}(\Omega^b;\mathbb R^3)\right\}\leq
K^{a,\infty}(u^a),
\end{align*}
for every $u^a \in BV(\;]0,1[;\RI^3))$, where $K^{a,\infty}$ is defined in \eqref{def:Ka}.
\end{Proposition}
\begin{proof}[Proof]
The proof relies on the same arguments exploited in the proof of  Proposition \ref{2ndestimateub}, namely on the lower semicontinuity of the functional in \eqref{probl_infty} with respect to the weak * convergence in $BV(\Omega^a;\mathbb R^3)$ (see Lemma \ref{preFlsc}, and the relaxation with respect to the same topology of the functional $F_a: c^{a,\infty}+ W^{1,1}_a(]0,1[;\mathbb R^3)\to \mathbb R^+$ defined as $\int_{\Omega^a} \hat W(\nabla_3 u^a(x_3)dx$). Indeed the latter relaxed functional coincides with $K^{a,\infty}$ in \eqref{def:Ka}, in view of \cite[Example 4.1]{BFMTraces}.
\end{proof}

\subsubsection{Case $\displaystyle\frac{h_n}{r_n} \to \infty$}

 In the next lemma we prove that the sequence of deformations $u^b_n(x_\alpha, x_3)-(x_\alpha, h_n x_3)$ on the thin film below admits the origin as a Lebesgue point. This result will be instrumental to deduce a penalization term in the final relaxed energy regarding the uncoupled behaviour of $u^a$ at the origin.

\begin{Proposition}
\label{traceconvergence}
    Let $(g_n)_{n\in\N}\subset W^{1,1}(\Omega^b;\RI^3)$ satisfy $g_n\to 0$ in $W^{1,1}(\Omega^b;\RI^3)$. Assume that,
    \begin{align}
    \label{growthalphaprop}
        &|\nabla_\alpha g_n|_{L^1(\Omega^b;\RI^3)} \leq C\frac{r_n^2}{h_n},  &\forall n\in \N,
        \\
        \label{growth3prop}
        & |\nabla_3 g_n|_{L^1(\Omega^b;\RI^3)} \leq C r_n^2   &\hbox{as } n \to \infty,
        \\
        \label{boundaryprop}
        &g_n|_{\partial\omega\x]-1,0[}= \bf0,  &\forall n\in \N,
        \\
        &\frac{h_n}{r_n} \to \infty, \quad r_n,h_n\to0 &\hbox{as } n \to \infty \label{inftyinfty}.
    \end{align}

    Then,
    $$
    \lim_{n\to \infty} \int_{\omega}|g_n|(r_nx_\alpha,0)dx_\alpha = 0.
    $$
\end{Proposition}

\begin{proof}

    Let $\delta>0$ and consider $n\in \N$ large enough satisfying
    \begin{align}
    \label{inclusion_omega}
       ]-\delta r_n,\delta r_n[^2 \; \subset   \omega \; &\subset \; ]-\delta,\delta[^2.
    \end{align}

 \noindent Recalling \eqref{boundaryprop} we can extend $g_n\in W^{1,1}_0(\Omega^b;\mathbb R^3)$ to $W^{1,1}_0(\RI^3;\mathbb R^3)$ as follows:
  \begin{equation}
      \tilde g_n = \left\{\begin{array}{ll}
        g_n &\hbox{ in } \Omega^b,\\
        \bf0 &\hbox{ in }  \RI^3 \setminus \Omega^b.
        \end{array}\right.
        \label{def:extension_gn}
  \end{equation}
  For easier notation we simply denote $g_n$ as this extension.

    \noindent From the absolute continuity on lines of a function in $W^{1,1}(\Omega^b;\RI^3)$, we can apply the Fundamental Theorem of Calculus on the lines between $(\overline{x}_1,\delta,\overline{x}_3)$ and $(\overline{x}_1,\overline{x}_2,\overline{x}_3)$ for a.e. $(\overline{x}_1,\overline{x}_2,\overline{x}_3)\in ]-\delta r_n,\delta r_n[^2\x]-1,0[$. 

    $$
    g_n(\overline{x}_1,\overline{x}_2,\overline{x}_3) = g_n(\overline{x}_1,\overline{x}_2,\overline{x}_3)-g_n(\overline{x}_1,\delta,\overline{x}_3) = \int_{\overline{x}_2}^{\delta} \frac{\partial}{\partial x_2}g_n(\overline{x}_1,y,\overline{x}_3) dy
    $$
    where in the first equality we notice that from \eqref{inclusion_omega} we have $(\overline{x}_1,\delta,\overline{x}_3)\in \RI^3 \setminus \Omega^b$ therefore from \eqref{def:extension_gn}\; $ g_n(\overline{x}_1,\delta,\overline{x}_3)=0$.

    We can now estimate $|g_n(\cdot,x_3)|$ in $]-\delta r_n,\delta r_n[^2$. For $(\overline{x}_1,\overline{x}_2)\in ]-\delta r_n,\delta r_n[^2$

    \begin{align}
    |g_n(\overline{x}_1,\overline{x}_2,\overline{x}_3)| \leq \int_{\overline{x}_2}^{\delta} \left|\frac{\partial}{\partial x_2}g_n(\overline{x}_1,y,\overline{x}_3)\right| dy &\leq \int_{\overline{x}_2}^{\delta} \left|\nabla_\alpha g_n(\overline{x}_1,y,\overline{x}_3)\right| dy 
    \leq \int_{-\delta r_n}^{\delta} \left|\nabla_\alpha g_n(\overline{x}_1,y,\overline{x}_3)\right| dy.
    \label{point_ineq}
    \end{align}

    By noticing that this last term doesn't depend on $\overline{x}_2$ we get that by integrating in $]-\delta r_n,\delta r_n[^2$ and using \eqref{inclusion_omega} in the first inequality

    \begin{align}
    \|g_n(\cdot,\overline{x}_3)\|_{L^1(r_n\omega)} \leq \|g_n(\cdot,\overline{x}_3)\|_{L^1(]-\delta r_n,\delta r_n[^2)} &\leq 2\delta r_n \int_{-\delta r_n}^{\delta r_n} \int^{\delta}_{-\delta r_n} \left|\nabla_\alpha g_n(x_1,x_2,\overline{x}_3)\right| dx_2dx_1.
    \label{inequality_x3_fixed}
    \end{align}

    Notice that these arguments are valid for $\overline{x}_3$ a.e.. Since from \eqref{def:extension_gn} $\nabla_\alpha g_n(x)=0$ in $]-\delta,\delta[^2 \setminus \omega \; \x ]-1,0[$, then integrating \eqref{inequality_x3_fixed} on $x_3$ we get

    \begin{align}
    \|g_n\|_{L^1(r_n\omega \x ]-1,0[)} \; \leq \; 2\delta r_n \|\nabla_\alpha g_n\|_{L^1(]-\delta r_n,\delta[\x]-\delta r_n,\delta r_n[ \x ]-1,0[)} &= 2\delta r_n \|\nabla_\alpha g_n\|_{L^1(r_n\omega\x]-1,0[)}
    \nonumber
    \\
    &\leq 2\delta r_n\|\nabla_\alpha g_n\|_{L^1(\O^b)}.\label{ineq:trace_conv} 
    \end{align}

Let $v_n (x_\alpha,x_3):=g_n(r_nx_\alpha,\rho_nx_3)$ be defined on $\Omega^b$, with $\rho_n\to0$ to be specified later. The idea is to prove the following three convergences
    \begin{equation}\label{prop_des_conv}
    \hbox{1) }v_n \to 0 \hin L^1\domb, \quad 
    \hbox{2) }\nabla_\alpha v_n \to 0 \hin L^1(\Omega; \mathbb {R}^{3\times 2}),\quad
    \hbox{3) }\nabla_3v_n \to 0 \hin L^1\domb.
    \end{equation}
    implying $v_n\to 0 \hin W^{1,1}\domb$ and by the continuity of the trace operator we get the desired result
    $$
    \lim_n \int_{\omega\times \{0\}}|g_n|(r_nx_\alpha,0)dx_\alpha = \lim_n \int_{\omega\times \{0\}}|v_n|dx_\alpha = 0.
    $$

    Proof of 1).
    \begin{align}\nonumber
        \int_{\Omega^b} |v_n|dx = \frac{1}{r_n^2\rho_n}\int_{r_n\o\x(-\rho_n,0)}|g_n|dx \leq C\frac{1}{r_n\rho_n}\|\nabla_\alpha g_n\|_{L^1(\Omega^b)} \leq C\frac{r_n}{h_n}\frac{1}{\rho_n}
    \end{align}
    where in the first inequality it was used \eqref{ineq:trace_conv} and in the last in the last inequality \eqref{growthalphaprop}.
    
    Proof of 2)
    \begin{align}\nonumber
        \int_{\Omega^b} |\nabla_\alpha v_n|dx = \frac{1}{r_n^2\rho_n}\int_{r_n\o\x(-\rho_n,0)}r_n|\nabla_\alpha g_n|dx \leq  C\frac{r_n}{h_n}\frac{1}{\rho_n}
    \end{align}
    where in the inequality it was used again \eqref{growthalphaprop}

    Proof of 3)
    \begin{align}\nonumber
        \int_{\Omega^b} |\nabla_3 v_n|dx = \frac{1}{r_n^2\rho_n}\int_{r_n\o\x(-\rho_n,0)}\rho_n|\nabla_3 g_n|dx \leq  C\rho_n
    \end{align}
    where it was used \eqref{growth3prop}.

    Choosing $\rho_n= \displaystyle\left(\frac{r_n}{h_n}\right)^{1/2}$ we get the desired convergences.
    \color{black}
    
\end{proof}

Based on the result just stated we can extend Proposition \ref{Compactness_infty}, specifically for $\frac{h_n}{r_n}\to \infty$. 

\begin{Lemma}[Compactness $(\infty, \infty)$] \label{lemma_conv_trace_above} Under the same assumptions of Proposition \ref{Compactness_infty} and assuming \eqref{inftyinfty} then,
for every sequence $( u_n^a, u_n^b) \in \mathcal{U}^\infty_n$  such that 
         \(\sup_{n\in\N} |F_n^\infty[(  u_n^a,  u_n^b)]|
<+\infty\),   it follows that
\eqref{convinftycomp} holds and 
$$
u_n^a(\cdot,0) \to {\bf 0} \hbox{ in } L^1(\omega;\RI^3).
$$
\end{Lemma}
\begin{proof}
    As in the other analogous results above, we assume null loads, referring to Remark \ref{remloads} for the general case. From Proposition \ref{Compactness_infty}, the assumptions on the boundary \eqref{dirichletinfty}, i.e. \begin{align*}
u_n^b|_{\partial\omega\x]-1,0[}=(x_\alpha,h_nx_3),
    \end{align*}
and the second condition in \eqref{convinftycomp},
    namely
    \begin{align*}
u_n^b \to (x_\alpha,0) \hbox{ in } W^{1,1}(\Omega^b;\RI^3),
    \end{align*}
we have that \eqref{growth_3_infty} and \eqref{growth_alpha_infty} hold, which, in turn, imply
    \begin{align*}
    \int_{{\Omega}^b}
\left|\nabla_{\alpha}u^b_n 
 \;  - \;Id_\alpha \right| dx \leq \frac{r_n^2}{h_n}C,
\\
\int_{\Omega^b} \left|\left( \nabla_{3}u^b_n 
 \;  - \;(0_\alpha,h_n) \right)\right|dx \leq C r_n^2
\end{align*}
with $Id_\alpha$ 
as in \eqref{matrixdef}.

Applying Lemma \ref{traceconvergence} to the function $g_n:=u_n^b-(x_\alpha,h_nx_3)$ and recalling that $( u_n^a, u_n^b) \in \mathcal{U}^\infty_n$ and in particular \eqref{eq:bc} holds, we get

\begin{align*}
\lim_{n \to \infty} \int_{\omega}|u_n^b(r_nx_\alpha,0)-(r_nx_\alpha,0)|dx_\alpha = &\lim_{n\to \infty} \int_{\omega}|u_n^a(x_\alpha,0)-(r_nx_\alpha,0)|dx_\alpha = 0,
\end{align*}
which implies
\begin{align*}
u_n^a(\cdot,0) \to \mathbf{0} \hbox{ in } L^1(\omega;\RI^3),
\end{align*}
that concludes the proof.
\end{proof}

In view of Lemma \ref{lemma_conv_trace_above}, we can rewrite problem \eqref{probl_infty} as 

\begin{align}\label{probl2_infty}
G^\infty(u^a, (x_\alpha, 0))=\inf\left\{\liminf_nF_n^\infty(u_n^a, u_n^b): (u_n^a, u_n^b)\in \mathcal U^\infty_n,  u_n^a \overset{\ast}{\rightharpoonup} u^a \hbox{ in }BV(\Omega^a;\mathbb R^3), \right. \nonumber
\\
\left.u_n^a(\cdot,0) \to {\bf 0} \hbox{ in } L^1(\omega;\RI^3),
u_n^b  \to (x_\alpha,0) \hbox{ in }W^{1,1}(\Omega^b;\mathbb R^3)\right\}.
\end{align}

The representation of Theorem \ref{generalrepresentation}, given by \eqref{probl_infty,infty}, will be achieved by proving a double inequality.

\begin{Proposition}[Lower bound $(\infty, \infty)$]
\label{prop_lb_infty2}
    Consider the Borel function  \(W : \mathbb{R}^{3\x 3}\to 
    \mathbb{R}\) satisfying  \eqref{coerci} and \eqref{eq:growth1}.  
    For  $n \in \mathbb N$ let the spaces
    \(\mathcal{U}_n^\infty\) be introduced in \eqref{Vn}
    where the Dirichlet conditions in $\U_n^\infty$ are as in \eqref{dirichletinfty}.
    Let $F_n^\infty$  be the energy functional defined in \eqref{eq:Fninfty},
    where \(r_n\), \(h_n\) and the loads $H_n^{a,\infty},H_n^{b,\infty}$ satisfy conditions \eqref{hrzero}, \eqref{loadinfty} and $\lim_{n\to \infty}\frac{h_n}{r_n}=\infty$.
    Assume also that the recession function of $\hat W$ \eqref{hatW} defined as in \eqref{S101rec}, satisfy \eqref{Whatrec}.
    
    Then for every $u^a \in BV(]0,1[;\RI^3)$,
    \begin{align*}\nonumber
\inf\left\{\liminf_nF_n^\infty(u_n^a, u_n^b): (u_n^a, u_n^b)\in \mathcal U^\infty_n,  u_n^a \overset{\ast}{\rightharpoonup} u^a \hbox{ in }BV(\Omega^a;\mathbb R^3), u_n^a(\cdot,0) \to {\bf 0} \hbox{ in } L^1(\omega;\RI^3),\right. 
\\
\left.
u_n^b  \to (x_\alpha,0) \hbox{ in }W^{1,1}(\Omega^b;\mathbb R^3)\right\} \geq K^{a,\infty}(u^a) + \hat{W}^{**} (u^a(0^+)),
\end{align*}
where $K^{a,\infty}$ is defined in \eqref{def:Ka}.
\end{Proposition}

\begin{proof}
The following functional $J^a:BV(]0,1[;\mathbb R^3)\to \mathbb R$ is clearly a lower bound to the left hand side of \eqref{probl2_infty}.

\begin{align}
J(u^a) := \inf\left\{\liminf_n \int_{\Omega^a} \hat{W}^{**}(\nabla_3 u_n^a(x))dx: u_n^a \in (W_a^{1,1}(\Omega^a;\RI^3) + c^{a,\infty} ),   \right. \nonumber
\\
\left.u_n^a \overset{\ast}{\rightharpoonup} u^a \hbox{ in }BV(\Omega^a;\mathbb R^3), u_n^a(\cdot,0) \to {\bf 0}  \hbox{ in } L^1(\omega;\RI^3)
\right\}.
\label{probl2_plate}
\end{align}

Now we prove $J(u^a)\geq K^a(u^a) + \hat{W}^{**} (u^a(0^+))$.
 Let $u^a \in BV(]0,1[;\mathbb R^3)$, and 
let $(u_n^a)_n$ be a recovery sequence for \eqref{probl2_plate}, 
i.e.

\begin{align}
    &u_n^a \in (W_a^{1,1}(\Omega^a;\RI^3) + c^{a,\infty} ), \quad \lim_{n\to \infty} \int_{\Omega^a} \hat{W}^{**}(\nabla_3 u_n^a(x))dx = J(u^a),\nonumber
    \\
    &u_n^a \overset{\ast}{\rightharpoonup} u^a \hbox{ in }BV(\Omega^a;\mathbb R^3), \quad u_n^a(\cdot,0) \to {\bf 0} \hbox{ in } L^1(\omega;\RI^3). \nonumber
\end{align}

Let $\Omega^a_\varepsilon := \omega\x]-\varepsilon,1[$ and $\Omega^{a-}_\varepsilon := \omega\x]-\varepsilon,0[$. Let also $\gamma_n\in L^1(\partial  \Omega^{a-}_\varepsilon;\RI^3)$ be defined as
\begin{equation}
\label{g_n_definition}
\gamma_n = \left\{\begin{array}{ll}
        u_n^a(\cdot,0) &\hbox{ in } \; \; \; \partial  \Omega^{a-}_\varepsilon \cap \{x_3=0\} ,\\
        {\bf 0} &\hbox{ in } \; \; \; \partial  \Omega^{a-}_\varepsilon \setminus \{x_3=0\} .
        \end{array}\right.
\end{equation}
By Gagliardo's trace theorem (see \cite{Gagl}), we can find a function $g_n \in W^{1,1}(\Omega^{a-}_\varepsilon;\RI^3)$ such that 
\begin{equation}
    \label{gagliardo_ineq}
    \hbox{tr} (g_n) = \gamma_n, \quad
    |g_n|_{W^{1,1}(\Omega^{a-}_\varepsilon)} \leq C |\gamma_n|_{L^1(\partial  \Omega^{a-}_\varepsilon)},
\end{equation}
for some $C=C(\Omega^{a-}_\varepsilon)>0$.

Consider now the following extensions:

\begin{align*}
\tilde{u}_n^a(x_\alpha,x_3) &= \left\{\begin{array}{ll}
        u_n^a(x_\alpha,x_3) &\hbox{ if } \; \;\;0<x_3 < 1,\\
        g_n(x_\alpha,x_3) &\hbox{ if } -\varepsilon<x_3 <0,
        \end{array}\right.
        \\
        \\
\tilde{u}^a(x_3) &= \left\{\begin{array}{ll}
        u^a(x_3) &\hbox{ if } \; \;\;0<x_3 < 1,\\
        {\bf 0} &\hbox{ if } -\varepsilon<x_3 <0.
        \end{array}\right.
\end{align*}

It is easily checked using \eqref{coerci}, \eqref{g_n_definition} and \eqref{gagliardo_ineq} that
:

\begin{align*}
    (\tilde{u}_n^a)_n \subset W^{1,1}(\Omega^a_\varepsilon;\mathbb R^3),
    \;
    \tilde{u}_n^a \to \tilde{u}^a \hbox{ in } L^1(\Omega^a_\varepsilon;\mathbb R^3),
    \;
    \sup_{n\in \N} |\nabla \tilde{u}_n^a|_{L^1(\Omega^a_\varepsilon)}  <  +\infty.
\end{align*} 

We have that $(\tilde{u}_n^a)_n$ is admissible to the following extended problem defined on $\Omega^a_\varepsilon$,

\begin{align*}
\inf\left\{\liminf_n \int_{\Omega^a_\varepsilon} \hat{W}^{**}(\nabla_3 u_n(x))dx: u_n \in (W_a^{1,1}(\Omega^a_\varepsilon;\RI^3) + c^{a,\infty} ),   
u_n \overset{\ast}{\rightharpoonup} \tilde{u}^a \hbox{ in }BV(\Omega^a_\varepsilon;\mathbb R^3)
\right\}.
\end{align*}

And so, arguing as in Proposition \ref{lb}, it results that

\begin{align}
\liminf_{n\to\infty}\int_{\Omega^a_\varepsilon} \hat{W}^{**}(\nabla_3 \tilde{u}_n^a(x))dx \geq
\int_{\Omega^a_\varepsilon}{\hat W}^{**}(
        \nabla_{3}u^a(x_3))dx\;  &+\; \int_{\Omega^a_\varepsilon}({\hat W}^{**})^{\infty}\left(\frac{d D^s_3 u^a}{d |D^s_3 u^a|},\right)d |D^s_3 u^a|\; +
       \nonumber
        \\
        & +\;(\hat W^{\ast \ast})^\infty (c^{a,\infty}- u^a(1^-))
\end{align}

Now we estimate from below \eqref{probl2_plate} using the inequality just stated.

\begin{alignat}{2}
    \liminf_n \int_{\Omega^a} {\hat W}^{**} (\nabla_3 u_n^a(x)) dx = 
    &\liminf_n \left(\int_{\Omega^a_\varepsilon} {\hat W}^{**} (\nabla_3 \tilde{u}_n^a(x)) dx - \int_{\Omega^a_\varepsilon \setminus \Omega^a} {\hat W}^{**} (\nabla_3 \tilde{u}_n^a(x)) dx\right) &\geq
    \nonumber
    \\
    &\liminf_n \int_{\Omega^a_\varepsilon} {\hat W}^{**} (\nabla_3 \tilde{u}_n^a(x)) dx - \limsup_n \int_{\Omega^a_\varepsilon \setminus \Omega^a} {\hat W}^{**} (\nabla_3 \tilde{u}_n^a(x)) dx &\geq
    \nonumber
    \\
    \nonumber
    \\
   \geq &\int_{\Omega^a_\varepsilon}{\hat W}^{**}(
        \nabla_{3}\tilde{u}^a(x_3))dx + \int_{\Omega^a_\varepsilon}({\hat W}^{**})^{\infty}\left(\frac{d D^s_3 \tilde{u}^a}{d |D^s_3 \tilde{u}^a|},\right)d |D^s_3 \tilde{u}^a| &+ 
        \nonumber
        \\
        &(\hat W^{\ast \ast})^\infty (c^{a,\infty}- u^a(1^-)) - \limsup_n \int_{\Omega^a_\varepsilon \setminus \Omega^a} {\hat W}^{**} (\nabla_3 \tilde{u}_n^a(x))  dx &=
    \nonumber
    \\
    \nonumber
    \\
    = &\int_{\Omega^a_\varepsilon}{\hat W}^{**}(
        \nabla_{3}\tilde{u}^a(x_3))dx + \int_{\Omega^a_\varepsilon}({\hat W}^{**})^{\infty}\left(\frac{d D^s_3 \tilde{u}^a}{d |D^s_3 \tilde{u}^a|},\right)d |D^s_3 \tilde{u}^a| &+ 
        \nonumber
        \\
        &(\hat W^{\ast \ast})^\infty (c^{a,\infty}- u^a(1^-)) - \int_{\Omega^a_\varepsilon \setminus \Omega^a} {\hat W}^{**} ({\bf 0}) dx &=
        \label{inequality_convergence2}
    \\
    \nonumber
    \\
    = &\int_{\Omega^a}{\hat W}^{**}(
        \nabla_{3}u^a(x_3))dx + \int_{\Omega^a}({\hat W}^{**})^{\infty}\left(\frac{d D^s_3 u^a}{d |D^s_3 u^a|},\right)d |D^s_3 u^a| &+ 
        \nonumber
        \\
        &(\hat W^{\ast \ast})^\infty (u^a(0^+)) +
        (\hat W^{\ast \ast})^\infty (c^{a,\infty}- u^a(1^-)) 
        \label{last_equality}
\end{alignat}

In \eqref{inequality_convergence2} we exploited Dominated Convergence theorem,  the continuity of $\hat W^{\ast \ast}$ and $\nabla_3\tilde{u}_n^a \to 0$ in $L^1(\Omega^a_\varepsilon \setminus \Omega^a;\RI^3)$, due to \eqref{probl2_plate}, the definition of $\tilde u^a_n$ and $\tilde u^a$, and \eqref{g_n_definition} together with \eqref{gagliardo_ineq}.

In \eqref{last_equality} we exploited that $\mathcal L^2(\omega)=1$ and
\begin{align*}
\int_{\omega\x]-\varepsilon,0]}({\hat W}^{**})^{\infty}\left(\frac{d D^s_3 u^a}{d |D^s_3 u^a|},\right)d |D^s_3 u^a|& = \int_{\omega\x\{0\}}({\hat W}^{**})^{\infty}\left(\frac{d D^s_3 u^a}{d |D^s_3 u^a|},\right)d |D^s_3 u^a|=
\\
&=(\hat W^{\ast \ast})^\infty (u^a(0^+))
\end{align*}

Consequently, by letting $\varepsilon\to 0$, from \eqref{last_equality} it results that 
$$
J(u^a) = \lim_{n\to \infty} \int_{\Omega^a} {\hat W}^{**} (\nabla_3 u_n^a(x)) dx \geq K^{a,\infty}(u^a) + (\hat W^{\ast \ast})^\infty (u^a(0^+)).
$$
\end{proof}

Now we are in position to prove the upper bound.

A construction similar to Proposition \ref{ub1stestimateU} gives rise to the proof of the following result, which is key for the attainment of the upper bound

\begin{Proposition}\label{ub_continuous_joined}

Under the same assumptions of Proposition \ref{prop_lb_infty2}
\begin{align*}
\inf&\left\{\liminf_nF_n^\infty(u_n^a, u_n^b): (u_n^a, u_n^b)\in \mathcal U_n^\infty,  u_n^a \overset{\ast}{\rightharpoonup} u^a \hbox{ in }BV(\Omega^a;\mathbb R^3), \right. 
\\
&\left.u_n^a(\cdot,0) \to {\bf 0} \hbox{ in } L^1(\omega;\RI^3),
u_n^b  \to (x_\alpha,0) \hbox{ in }W^{1,1}(\Omega^b;\mathbb R^3)\right\} \leq
\int_0^1\hat{W}(\nabla u^a(x_3))dx,
\end{align*}
for every $u^a \in c^{a,\infty}x_3+ W^{1,1}_0(]0,1[;\RI^3)$. 
\end{Proposition}

\begin{proof}[Proof]
Considering the following sequences and with the arguments used in Proposition \ref{ub1stestimateU} we get the desired result.

For $z^a\in W^{1,1}(]0,1[;\RI^{3\x2})$,   recalling that, by Lemma \ref{lemma_conv_trace_above}, the limit function of every converging sequence $(u^b_n)_n$ above, is $u^b=(x_\alpha, 0)$,
\begin{align*}
u_n^a(x_\alpha, x_3):=\left\{ \begin{array}{ll} [z^a (\varepsilon_n )r_n x_\alpha
+ u^a(\varepsilon_n)] \frac{x_3}{\varepsilon_n} + u^b(r_n x_\alpha) \frac{\varepsilon_n-x_3}{\varepsilon_n} &\hbox{ if } (x_\alpha, x_3)\in \omega \times ]0,\varepsilon_n[,\\
z^a(x_3) r_n x_\alpha + u^a(x_3) &\hbox{ if }(x_\alpha, x_3) \in \omega \times [\varepsilon_n, 1[,
\end{array}\right.
\end{align*}
\begin{align*}
u^b_n(x_\alpha, x_3):= (x_\alpha,h_nx_3) \hbox{ if } (x_\alpha, x_3) \in \Omega^b. 
\end{align*}
Indeed it is easily seen that $(u_n^a, u^b_n) \in \mathcal U^{\infty}_n$. Furthermore 
\eqref{convubaquote}, \eqref{convubbquote} and \eqref{limenquote} hold, with $K^{a,q}_n$ and $K^{b,q}_n$ replaced by $K^{a,\infty}_n$ and $K^{b,\infty}_n$, respectively, in the last equality.
\end{proof}

Notice that $c^{a,\infty}x_3 + W^{1,1}_0(]0,1[;\RI^3) =  \{ u \in W^{1,1}(]0,1[;\RI^3): \; u({\bf 0})={\bf 0} \hbox{ and } u(1) = c^{a,\infty}\}$. When relaxing to BV and applying \cite[Example 4.1]{BFMTraces} 
we get the subsequent result.

\begin{Proposition}[Upper bound $(\infty, \infty)$]\label{ub_joined}
Under the same assumptions as Proposition \ref{prop_lb_infty2}, 
\begin{align*}
\inf\left\{\liminf_nF_n^\infty(u^a_n,u^b_n): (u_n^a, u_n^b)\in \mathcal U^\infty_n,  u_n^a \overset{\ast}{\rightharpoonup} u^a \hbox{ in }BV(\Omega^a;\mathbb R^3), u_n^a(\cdot,0) \to {\bf 0} \hbox{ in } L^1(\omega;\RI^3),\right. 
\\
\left.
u_n^b  \to (x_\alpha,0) \hbox{ in }W^{1,1}(\Omega^b;\mathbb R^3)\right\} \leq
K^{a,\infty}(u^a) + (\hat W^{\ast \ast})^\infty (u^a(0^+)).
\end{align*}
\end{Proposition}
\begin{proof}
The proof relies on the same arguments exploited in the proof of  Proposition \ref{ub_infty,0}, namely on the lower semicontinuity of the functional in \eqref{probl2_infty} with respect to the weak * convergence in $BV(\Omega^a;\mathbb R^3)$ (see Lemma \ref{preFlsc}, and the relaxation with respect to the same topology of the functional $F_a: c^{a,\infty}x_3+ W^{1,1}_a(]0,1[;\mathbb R^3)\to \mathbb R^+$ defined as $\int_{\Omega^a} \hat W(\nabla_3 u^a(x_3))dx$). Indeed the latter relaxed functional coincides with $K^{a,\infty}(u^a) + (\hat W^{\ast \ast})^\infty (u^a(0^+))$, in view of \cite[Example 4.1]{BFMTraces}.

\end{proof}

\begin{Remark}
    \label{general_dirichlet2}
    Notice that in this section, either when $h_n/r_n \to \infty $ or $h_n/r_n\to 0$, the arguments regarding the "above" functions $(u_n^a)_n$ in the upper bounds and lower bounds are similar to the first case, $\ell=q\in]0,+\infty[$. Notice also that Remark \ref{general_dirichlet1} applies to the lower bound Propositions \ref{lbinfty0} and \ref{prop_lb_infty2}. For the upper bounds in Propositions \ref{ub_infty,0} and \ref{ub_continuous_joined} we note that if in the construction of the recovery sequences we replace $z^a$ by $z^a+d^{a,\infty}$ with $z^a\in W^{1,1}_0(]0,1[;\RI^{3\x2})$ and a general $d^{a,\infty}\in\RI^{3\x2}$, we achieve the same results.
\end{Remark}

\subsection{Case $l = 0$}

We now consider the case \eqref{load0}. For the same reasons as in Remark \ref{remloads}, we can neglect the term $\frac{r_n^2}{h_n}\int_{\Omega^a} H_n^{a,0}(x) \cdot u^a_n(x)dx + \int_{\Omega^b} H_n^{b,0}(x)\cdot u^b_n(x)dx$ based on \eqref{load0}. We can therefore also assume $H_n^{a,0}=0$ and $H_n^{b,0}=0$. 

\begin{Proposition}\label{Compactness_0}
       Consider the Borel function  \(W : \mathbb{R}^{3\x 3}\to 
    \mathbb{R}\) satisfying \eqref{coerci} and \eqref{eq:growth1}.  
    For  $n \in \mathbb N$ let \(r_n\), \(h_n\) be as in \eqref{hrzero}, with $\ell$ in \eqref{ell} coinciding with $0$. 
    For  $n \in \mathbb N$ let the spaces
    \(\mathcal{U}_n^0\) be introduced in \eqref{Vn}
    where the Dirichlet conditions in $\U_n^0$ are as in \eqref{dirichlet0}.
    Let $F_n^0$  be the energy functional defined in \eqref{eq:Fn0},
 and the loads $H_n^{a,0},H_n^{b,0}$ satisfy  \eqref{load0}.

    Then, for every  
        $( u_n^a, u_n^b) \in \mathcal{U}_n^0$  such that \(\sup_{n\in\N} |F_n^0(  u_n^a,  u_n^b)|
    <+\infty\), there exist an increasing sequence of positive integer
        numbers $(n_i)_{i \in \mathbb N}$,
        ${u}^b
        \in BV(\omega;\RI^3)$,  
        depending possibly on
        the selected subsequence $(n_i)_{i \in \mathbb N}$, such that
        \begin{equation*}\left\{
                \begin{array}{l}
                           u^a_{n_i} \to (0_\alpha,x_3) \hbox{ in
                        }W^{1,1}(\Omega^b;\RR^3),\\\\
                        u^b_{n_i}\overset{\ast}{\rightharpoonup}  {u}^b  \hbox{ in
                        }BV(\Omega^a;\RR^3),
                        \end{array}\right.\end{equation*}\medskip
    
         \begin{equation*}
                        \displaystyle{\frac{1}{r_{n_i}}}\nabla_{\alpha}  u^a_{n_i}
                        \to  
                        Id_\alpha
                        \hbox{  in
                        } L^1(\Omega^b;\mathbb R^{3\x2}).
        \end{equation*}\medskip
with $Id_\alpha\in\RI^{3\x2}$ as in \eqref{matrixdef}
\end{Proposition}

\begin{proof}
    The proof is omitted since it is identical to the one of Proposition \ref{Compactness_infty}.
\end{proof}

In view of Proposition \ref{Compactness_0}, for $u^b \in BV(\omega; \mathbb R^3)$, 
the left hand side of \eqref{prob_0} can be written as  
\begin{align*}
G^0((0_\alpha,x_3),u^b) = \inf&\left\{\liminf_nF_n^0(u_n^a, u_n^b): (u_n^a, u_n^b)\in \mathcal U_n^0,  \right. \\
&\left.u_n^a  \to (0_\alpha,x_3) \hbox{ in }W^{1,1}(\Omega^a;\mathbb R^3), 
u_n^b \overset{\ast}{\rightharpoonup} u^b \hbox{ in }BV(\Omega^b;\mathbb R^3) \right\}.
\end{align*}
where $G^0$ is defined in \eqref{normal_inf}.

\begin{Proposition}[Lower bound, $0$]
    \label{lbound_0}
    Consider the Borel function  \(W : \mathbb{R}^{3\x 3}\to 
    \mathbb{R}\) satisfying \eqref{coerci} and \eqref{eq:growth1}.  
    Let  \(r_n\), \(h_n\) and the loads $H_n^{a,0},H_n^{b,0}$ satisfy conditions \eqref{hrzero}, \eqref{ell}, with $\ell=0$, and \eqref{load0} respectively.

    For every $n \in \mathbb N$, let the spaces
    \(\mathcal{U}_n^0\) be introduced in \eqref{Vn}
    where the Dirichlet conditions in $\U_n^0$ are as in \eqref{dirichlet0}.
    Let $F_n^0$  be the energy functional defined in \eqref{eq:Fn0}. 
    Assume also that the recession function of $W_0$ \eqref{W0}, defined as in \eqref{S101rec}, satisfies \eqref{W0rec}.

    Then, 
    \begin{align}\label{problplate0}
        \inf&\left\{\liminf_nF_n^0(u_n^a, u_n^b): (u_n^a, u_n^b)\in \mathcal U_n^0, u_n^a  \to (0_\alpha,x_3) \hbox{ in }W^{1,1}(\Omega^a;\mathbb R^3), 
        u_n^b \overset{\ast}{\rightharpoonup} u^b \hbox{ in }BV(\Omega^b;\mathbb R^3) \right\} \geq K^{b,0}(u^b),
    \end{align}
    for every $u^b\in BV(\Omega^b;\RI^3)$,
    where $K^{b,0}$ is defined as in \eqref{def:Kb}.
\end{Proposition}

\begin{proof}
The relaxed energy \eqref{problplate0} has a clear lower bound 

\begin{align}\label{probl_plate}
\inf\left\{\liminf_n
K^{b,0}_n(  u^b_n): u_n^b \in ( W_b^{1,1}(\Omega^b;\RI^3) + f^{b,0}),\; 
u_n^b \overset{\ast}{\rightharpoonup} u^b \hbox{ in }BV(\Omega^b;\mathbb R^3) \right\}.
\end{align}
where $f^{b,0}$ is as in \eqref{dirichlet0}, and extended as a function of $\Omega^b$, constantly in the $x_3$ direction. 
By the same token as in Proposition \ref{lb}, $K^{b,0}(u^b)$ is a lower bound for \eqref{probl_plate} which implies that it is a lower bound also for \eqref{prob_0}.

\end{proof}

\begin{Proposition}[Upper bound $0$]\label{ub_0_continuous}
Under the same assumptions of Proposition \ref{lbound_0}, 
\begin{align*}
\inf&\left\{\liminf_nF_n^0(u_n^a, u_n^b): (u_n^a, u_n^b)\in \mathcal U_n, u_n^a  \to (0_\alpha,x_3) \hbox{ in }W^{1,1}(\Omega^a;\mathbb R^3),  \right. \nonumber\\
        &\left.
        u_n^b \overset{\ast}{\rightharpoonup} u^b \hbox{ in }BV(\Omega^b;\mathbb R^3) \right\} \leq \int_{\Omega^b} W_0(\nabla_\alpha u^b(x_\alpha)) dx,
        \end{align*}
        for every $u^b \in f^{b,0}+  C^\infty_0(\omega;\RI^3)  $.
\end{Proposition}

\begin{proof}
    The construction of the recovery sequence is similar to the construction in Proposition \ref{ub_sobolev_infty}.

     For $z^b\in W^{1,1}_0(\omega;\RI^{3})$ consider the following sequences:

\begin{alignat*}{3}
    &u_n^a(x_\alpha,x_3) &&= (r_nx_\alpha,x_3), &&(x_\alpha,x_3)\in \Omega^a,\\
    &u_n^b(x_\alpha,x_3) &&= (u^b+z^bh_nx_3)\cdot(1-\phi_n(x_\alpha)) \; + \; u_n^a\left(\frac{1}{r_n}x_\alpha, -h_nx_3\right)\cdot\phi_n(x_\alpha),  &&
    \\
    & &&=  (u^b+z^bh_nx_3)\cdot(1-\phi_n(x_\alpha)) \; + \;  (x_\alpha,-h_nx_3)\phi_n(x_\alpha) &&(x_\alpha,x_3)\in\Omega^b, 
\end{alignat*}
where, for every $n$, $\phi_n$ is the cut-off function defined in \eqref{cutoff} satisfying:

\begin{align*}
&\phi_n(x_\alpha):= \left\{\begin{array}{ll}
        1 &\hbox{ if } x_\alpha\in r_n\omega,\\
        0 &\hbox{ if } x_\alpha \notin (1+\beta)r_n\omega, 
        \end{array}\right. \\ 
&|\nabla\phi_n(x_\alpha)| \leq \frac{C}{\beta r_n} \quad \hbox{and} \quad 0 \leq \phi_n(x_\alpha) \leq 1.
\end{align*}
Here $\beta > 0$ is a fixed arbitrary value.
It results that
\begin{align*}
    \nabla_\alpha u_n^b(x_\alpha, x_3)  =& \; [\nabla_\alpha u^b(x_\alpha) + \nabla_\alpha z^b(x_\alpha)h_nx_3](1-\phi_n(x_\alpha)), 
    \\
    &+ \nabla_\alpha\phi_n(x_\alpha) \otimes \left[ (x_\alpha,-h_nx_3) - (u^b(x_\alpha)+z^b(x_\alpha)h_nx_3 \right]  + \phi_n(x_\alpha)Id_\alpha,
\\
    \nabla_3 u_n^b(x_\alpha, x_3) =& h_n\left(z^b(x_\alpha)-\left(Id_3+ z^b(x_\alpha)\right)\phi_n(x_\alpha)\right).
\end{align*}
with $Id_\alpha\in\RI^{3\x2}$ as in \eqref{matrixdef}. 

Hence,
$$
\frac{1}{h_n}\nabla_3u_n^b \to z^b
\hbox{ in } L^1(\Omega^b;\RI^3).
$$

Defining, as above, $B_n:=(1+\beta)r_n\omega\times]-1,0[$ and $A_n:=r_n\omega\times]-1,0[$, 
\begin{align}
    &\int_{\Omega^b} | \nabla_\alpha u_n^b - \nabla_\alpha u^b| dx \leq
    \nonumber
    \\
    \nonumber
    &\int_{\Omega^b\setminus A_n} |\nabla_\alpha z^bh_nx_3|dx + \frac{C}{r_n}\int_{B_n\setminus A_n}\left[|(x_\alpha,-h_nx_3)| + |u^b|+|z^bh_nx_3|\right]dx \; + 
    \\
    +&\int_{B_n}\left|\phi_n(x_\alpha)Id_\alpha\right| + |\nabla_\alpha u^b|dx
\label{inequality_convergence_gradient0}
\end{align}

Because $|B_n| = O(r_n^2)$ and for $n\in \N$ large enough $u_b$ is bounded in $L^\infty(B_n;\mathbb R^3)$ (since $f^{b,0}$ satisfies \eqref{f=0=g} and $u^b \in f^{b,0}+ C^\infty_0(\omega;\RI^3) $). then we have that \eqref{inequality_convergence_gradient0} tends to $0$ as $n\to\infty$, i.e.

$$
\nabla_\alpha u_n^b \to \nabla_\alpha u^b \hbox{ in } L^1(\Omega^b;\RI^3).
$$

By the continuity of $W$, the dominated convergence theorem and $W(Id)=0$ we get:

\begin{align*}
    \lim_{n\to \infty}\left(\frac{r_n^2}{h_n}K^{a,0}_n(  u^a_n)+
K^{b,0}_n(  u^b_n)\right) =& \lim_{n\to \infty} \int_{\Omega^b} W\left(\nabla_\alpha u_n^b(x) \;\; | \;\; \frac{1}{h_n}\nabla_3 u_n^b(x)\right) dx= 
\\
=&\int_{\Omega^b} W(\nabla_\alpha u^b(x_\alpha) \;\; | \;\; z^b(x_\alpha)
)dx
\end{align*}
Once again by the arguments in \cite[Proposition~7]{LeDR1} we get the desired result by infimizing in $z^b\in W^{1,1}_0(\omega;\RI^{3})$.

\end{proof}

\begin{Proposition}[Upper bound $0$]
    Under the same assumptions as Proposition \ref{lbound_0} we have 

    \begin{align*}
        \inf&\left\{\liminf_nF_n^0(u_n^a, u_n^b): (u_n^a, u_n^b)\in \mathcal U_n^0, u_n^a  \to (0_\alpha,x_3) \hbox{ in }W^{1,1}(\Omega^a;\mathbb R^3),  \right. \nonumber\\
        &\left.
        u_n^b \overset{\ast}{\rightharpoonup} u^b \hbox{ in }BV(\Omega^b;\mathbb R^3) \right\} 
        \leq
        K^{b,0}(u^b),
\end{align*}
for every $u^b\in BV(\omega;\RI^3)$,
where $K^{b,0}$ is defined in \eqref{def:Kb}. 
\end{Proposition}
\begin{proof}
    From Proposition \ref{ub_0_continuous}, by applying a similar reasoning as in Proposition \ref{ub_infty,0} or Proposition \ref{2ndestimateub} we get the desired result.
\end{proof}

\begin{Remark}
    \label{general_dirichlet3}
    We note that, similarly to Remark \ref{general_dirichlet2}, in this section the arguments regarding the "below" functions $(u_n^b)_n$ in the upper bound and lower bound results are similar to the first case, $\ell=q\in ]0,+\infty[$. Remark \ref{general_dirichlet1} also applies to Proposition \ref{lbound_0}. For Proposition \ref{ub_0_continuous} we note that if in the construction of the recovery sequence we replace $z^b$ by $z^b+g^{b,0}$ with $z^b\in W^{1,1}_0(\omega;\RI^{3})$ and a general $g^{b,0}\in W^{1,1}(\omega;\RI^3)$ then we achieve the same results.
\end{Remark}

\section{The super-linear case}
For the sake of completeness we comment here on the case of super-linear growth for  $\ell= +\infty$ and $\ell= 0$, cases that were not addressed in \cite{GZNODEA}.
We underline that our subsequent analysis is quite complete in the case $\ell=\infty$, in the sense that all the $p>1$ are taken into account, while in the case $\ell=0$, due to the adopted techniques we are able to consider only $1<p<2$.

For $p >1$, assume that $W$ is a Borel function satisfying the following hypotheses:
\begin{equation}\label{coercip}
 W(M) \leq C(1+|M|^p),\quad\forall M\in \mathbb{R}^{3\x
3}
\end{equation}
for some $C > 0$
and
\begin{equation}\label{growthp}
   \frac{1}{C}|M-Id|^p\leq W(M) \quad \forall M \in \mathbb{R}^{{3\x 3}}, \quad \; \;  W(Id) = 0.
\end{equation}

\subsection{Case $\ell = +\infty$}

We start defining

\begin{align}\label{Upinftyn}\mathcal{U}^{\infty,p}_n=\Big\{ &(u^a, u^b) \in \left(c^{a,\infty}
+W_a^{1,p}
(\Omega^a;\RR^3)\right)\x \left(f^{b,\infty}+h_n g^{b,\infty}x_3
+W^{1,p}_b(\Omega^b;\RR^3)\right):
\\ &\qquad 
u^a \text{ and } u^b \text{ satisfy \eqref{eq:bc}}\Big\} \nonumber
\end{align}
with  the boundary data given by \eqref{dirichletinfty} and $W_a^{1,p}(\Omega^a;\mathbb
R^3)$  the
closure, with respect to $W^ {1,p}$-norm, of
 $$\left\{u^a\in C^\infty(\overline{\Omega^a};\RR^3)\,:\,
 u^a={\bf 0} \hbox{ in a  neighbourhood of }\omega\x\{1\}\right\},$$
 and
  $W_b^{1,p}(\Omega^b;\RR^3)$  the closure, with respect to $W^{1,p}$-norm,
of
 $$\big\{u^b\in C^\infty(\overline{\Omega^b};\RR^3)\,:\,
 u^b={\bf 0} \hbox{ in a neighbourhood of }\partial\omega\times]-1,0[\big\}.$$ 
Moreover, we define
\begin{align}
    \label{eq:Fninftyp}
    F_n^{\infty,p}(u_n^a, u_n^b) =   K_n^{a,\infty}(u_n^a) +\frac{h_n}{r_n^2} K_n^{b,\infty}(u_n^b), \quad (u_n^a,u_n^b) \in \mathcal{U}^{\infty,p}_n, \end{align}
where $K_n^{a,\infty}, K_n^{b,\infty}$ are defined by \eqref{kna} and \eqref{knb}, taking into account that the loads are
\begin{align}
H_n^{a,\infty}\rightharpoonup H \hbox{ in } L^{p'}(\Omega^a;\mathbb R^3), \hbox{ and } \frac{h_n}{r^2_n} H_n^{b,\infty} \rightharpoonup 0\hbox{ in }L^{p'}(\Omega^b;\mathbb R^3) 
\label{rateinftyp}
\end{align}
where $H\in L^{p'} (\Omega;\mathbb R^{3})$, with $\frac{1}{p}+ \frac{1}{p'}=1$.

In this case, the compactness result becomes 
\begin{Proposition}\label{Compactnessp_infty}
    Consider the Borel function  \(W : \mathbb{R}^{3\x 3}\to 
    \mathbb{R}\) satisfying  \eqref{coercip} and \eqref{growthp}.  
    For  $n \in \mathbb N$ let \(r_n\), \(h_n\) be as in \eqref{hrzero}, with $\ell$ in \eqref{ell} coinciding with $+\infty.$ Let the spaces
    \(\mathcal{U}_n^{p,\infty}\) be as in \eqref{Upinftyn} 
    with Dirichlet conditions as in \eqref{dirichletinfty}, and the loads $H_n^{a,\infty},H_n^{b,\infty}$ satisfy \eqref{rateinftyp}.
    Let $F_n^{\infty,p}$  be the energy functional defined in \eqref{eq:Fninftyp}.
 Then, for every  
        $( u_n^a, u_n^b) \in \mathcal{U}^{\infty,p}_n$  such that \(\sup_{n\in\N} |F_n^{\infty,p}(  u_n^a,  u_n^b)|
<+\infty\), there exist an increasing sequence of positive integer
        numbers $(n_i)_{i \in \mathbb N}$,
        ${u}^a
        \in W^{1,p}(]0,1[;\RI^3)$ 
        depending possibly on
        the selected subsequence $(n_i)_{i \in \mathbb N}$, such that
             \begin{equation*}
             \left\{
                \begin{array}{l}
                    u^a_{n_i}\rightharpoonup  {u}^a  \hbox{ in
                        }W^{1,p}(\Omega^a;\RR^3),\\\\  u^b_{n_i} \to (x_\alpha,0) \hbox{ in
                        }W^{1,p}(\Omega^b;\RR^3),\end{array}\right.\end{equation*}
        \noindent and, it results
         \begin{equation*}
         \displaystyle{\frac{1}{h_{n_i}}}\nabla_3  u^b_{n_i}
                        \to  (0_\alpha,1) \hbox{  in
                        }L^p(\Omega^b;\mathbb R^3).
        \end{equation*}  
        Furthermore, if $p>2$, $u^a(0)={\bf 0}$.
        Moreover, for every $p>1$, if $\lim_{n\to \infty}\frac{h_n}{r_n}= \infty$, then $\lim_{n\to \infty} u^a_n(\cdot, 0) = {\bf 0}$ in $L^p(\omega;\mathbb R^3)$.
\end{Proposition}
\begin{proof}
    The proof is very similar to the other compactness results proven above, exploiting also the same arguments as in \cite[Proposition 2.1]{GGLM1} to obtain the null condition at the origin for $u^a$.

    The last convergence is obtained from an adaptation of Proposition \ref{traceconvergence}. Following the proof of Proposition \ref{Compactness_infty} with the super-linear assumptions we get as counterparts of \eqref{growth_alpha_infty}  and \eqref{growth_3_infty} the following,
    \begin{align*}
        \int_{{\Omega}^b}
        \left|\nabla_{\alpha}u^b_n(x) 
         \;  - \;Id_\alpha \right|^p dx \leq \frac{r_n^2}{h_n}C
        \\
        \int_{{\Omega}^b}
        \left|\frac{1}{h_n}\nabla_{3}u^b_n(x) 
         \;  - \;(0_\alpha,1)\right|^p dx \leq \frac{r_n^2}{h_n}C,
    \end{align*}
    Based on these inequalities, we follow the arguments in Proposition \ref{traceconvergence} with the new assumptions
    \begin{align}
    \label{growthalphapropp}
        &\|\nabla_\alpha g_n\|_{L^p(\Omega^b;\RI^3)}^p \leq C\frac{r_n^2}{h_n},  &\forall n\in \N,
        \\
        \label{growth3propp}
        &\|\nabla_3 g_n\|_{L^p(\Omega^b;\RI^3)}^p \leq Cr_n^2h_n^{p-1},  &\forall n\in \N,
        \\
        \label{boundarypropp}
        &g_n|_{\partial\omega\x(-1,0)}={\bf 0},  &\forall n\in \N,
        \\
        &\frac{h_n}{r_n} \to \infty, \quad r_n,h_n\to0 &\hbox{as } n \to \infty \label{inftyinftyp}.
    \end{align}
    After reaching the \eqref{point_ineq}, we apply Jensen's inequality and rest of the proof follows with the following inequality instead of \eqref{point_ineq}
    $$
    |g_n(\overline{x}_1,\overline{x}_2,\overline{x}_3)|^p 
    \leq \int_{-\delta r_n}^{\delta} \left|\nabla_\alpha g_n(\overline{x}_1,y,\overline{x}_3)\right|^p dy.
    $$
    This time, for $v_n (x_\alpha,x_3):=g_n(r_nx_\alpha,x_3)$, the equivalent to the proofs of 1), 2) and 3) become
    \begin{alignat}{2}
        \label{vnp_estimate}
        &\int_{\Omega^b} |v_n|^pdx = \frac{1}{r_n^2}\int_{r_n\omega\x(-1,0)}|g_n|^pdx \leq C\frac{1}{r_n}\|\nabla_\alpha g_n\|_{L^p(\Omega^b)}^p &&\leq C\frac{r_n}{h_n}
        \\
        \nonumber
        &\int_{\Omega^b} |\nabla_\alpha v_n|^pdx = \frac{1}{r_n^2}\int_{r_n\omega\x(-1,0)}r_n^p|\nabla_\alpha g_n|^pdx &&\leq  C\frac{r_n^p}{h_n}
        \\
        \nonumber
        &\int_{\Omega^b} |\nabla_3 v_n|^pdx = \frac{1}{r_n^2}\int_{r_n\omega\x(-1,0)}\rho_n^p|\nabla_3 g_n|^pdx &&\leq  Ch_n^{p-1}
    \end{alignat}
    and so $v_n\to \mathbf{0}$ in $W^{1,p}$ if $\frac{h_n}{r_n}\to+\infty$. The desired convergence $\lim_{n\to \infty} u^a_n(\cdot, 0) = {\bf 0}$ in $L^p(\omega;\mathbb R^3)$ follows from the same arguments as Proposition \ref{lemma_conv_trace_above}.
    \color{black}
\end{proof}

The following result holds
\begin{Proposition}[$\Gamma$ - limit $(p,\infty)$]
\label{prop_p_infty}
    Consider the Borel function  \(W : \mathbb{R}^{3\x 3}\to 
    \mathbb{R}\) satisfying \eqref{coercip} and \eqref{growthp}.  
    For  $n \in \mathbb N$ let the spaces
    \(\mathcal{U}_n^\infty\) be introduced in \eqref{Upinftyn}
    where the Dirichlet conditions in $\U_n^{\infty,p}$ are as in \eqref{dirichletinfty}.
    Let $F_n^{\infty,p}$  be the energy functional defined in \eqref{eq:Fninftyp},
    where \(r_n\), \(h_n\) satisfy conditions \eqref{hrzero} with $\ell=\infty$ and the loads $H_n^{a,\infty}, H_n^{b,\infty}$ satisfy conditions \eqref{rateinftyp}.

    Then, if $1<p<2$ and and $\lim_{n\to \infty} \displaystyle \frac{h_n}{r_n^p} = 0$ \color{black} it results that for every $u^a \in c^{a,\infty}+W^{1,p}_a(]0,1[;\RI^3)$ 

    \begin{align}\nonumber
G^{\infty,p}(u^a, (x_\alpha, 0)):=\inf&\left\{\liminf_nF_n^{\infty,p}(u_n^a, u_n^b): (u_n^a, u_n^b)\in \mathcal U^{\infty,p}_n,  u_n^a \rightharpoonup u^a \hbox{ in }W^{1,p}(\Omega^a;\mathbb R^3), \right. \nonumber
\\
&\left.
u_n^b  \to (x_\alpha,0) \hbox{ in }W^{1,p}(\Omega^b;\mathbb R^3)\right\} = \displaystyle{\int_{]0,1[}{\hat W}^{**}\left(
        \nabla_{3}u^a(x_3)\right)}dx +\int_{\Omega^a} H(x)\cdot u^a(x_3) dx, \label{lowerbound_infty2p}
    \end{align}
    
    On the other hand, if either $p>2$ (with no assumption on the ratio $h_n/r_n$) or $1<p\leq2$ and $\lim_{n\to \infty}\frac{h_n}{r_n}=\infty$, then, the same representation holds for every $u^a \in c^{a,\infty}x_3 + W^{1,p}_0(]0,1[;\RI^3)$. 
\end{Proposition}
\begin{proof}[Proof] As in the case $p=1$ we assume that $H={\bf 0}$. Then, it is easily observed that the functional
$\int_{]0,1[}{\hat W}^{**}\left(
        \nabla_{3}u^a(x_3)\right)dx$
       is trivially a lower bound for the left hand side energy in \eqref{lowerbound_infty2p}, due to standard lower semicontinuity  results in the Sobolev setting, see \cite{DA}.

For what concerns the upper bound in the first case, i.e. if $1<p<2$ 
and $\lim_{n\to \infty} \frac{h_n}{r_n^p} = 0$
one can argue as in the case $p=1$ and $(\infty, 0)$, Proposition \ref{ub_sobolev_infty}. Estimates analogous to \eqref{int1} and \eqref{int2} hold, except that this time, the absolute values of the integrands should have an exponent $p$, and in particular in \eqref{int1}, instead of $\frac{1}{\beta r_n}$ we have $\left(\frac{1}{\beta r_n}\right)^p$ so after the change of variables, we conclude that \eqref{int1} vanishes if $h_n/r_n^p \to 0$.
\color{black}

The same conclusion to get an upper bound of the type $\int_0^1 \hat W(\nabla_3u^a(x_3))dx_3$ can be easily achieved in $c^{a,\infty}+ W^{1,p}_a(]0,1[;\mathbb R^3)$. To get the desired upper bound we apply a standard relaxation.

As regards the upper bound in the cases $\ell=+\infty$ and $p>2$ or $\ell=+\infty$, $\lim_{n\to \infty}\frac{h_n}{r_n}=\infty$ for every $1<p\leq 2$, we can argue as in the proof of Proposition \ref{ub_continuous_joined} to get
\begin{align*} G^{\infty,p}(u^a, (x_\alpha, 0))
\leq \int_0^1 \hat W(\nabla_3 u^a(x_3))dx_3
\end{align*}
for every $u^a \in c^{a,\infty}x_3 + W^{1,p}_0(]0,1[;\mathbb R^3)$.
To conclude it is enough to apply once again standard relaxation results for $\int_0^1 \hat W(\nabla_3 u^a(x_3))dx_3$ in $c^{a,\infty}x_3 + W^{1,p}_0(]0,1[;\mathbb R^3)$ (see \cite{DA}) and exploit the weak lower semicontinuity in $W^{1,p}(]0,1[;\mathbb R^3)$ of $G^{\infty,p}$. 
\end{proof}
\begin{Remark}
    It is important to note that the intermediate case for $1<p<2$ when $\displaystyle\lim_n\frac{h_n}{r_n^p}\neq 0$ and $\displaystyle\lim_n\frac{h_n}{r_n}\neq+\infty$ isn't covered in the last proof. 
    
    We also note that if in the limit we get the junction condition $u^a(0)=u^b(0_\alpha)={\bf 0}$ then it is easy to get the integral representation for $\F^{\infty,p}(u^a,(x_\alpha,0))$.

    The idea to cover these missing cases is to improve the compactness Proposition \ref{Compactnessp_infty}, in particular it would be desirable that the condition $\lim_n \frac{h_n}{r_n^p}=\infty$ implies $u^a(0)={\bf 0}$. 
    
    Note that for this to happen, it is only needed to improve the bound in the proof (of Proposition \ref{Compactnessp_infty}) of $v_n$ \eqref{vnp_estimate} to $C r_n^p/h_n$. Then the proof would be concluded since the $\nabla v_n$ already behaves in this way. Unfortunately, in this work it wasn't possible to achieve this upper bound.
    \color{black}
\end{Remark}

\subsection{Case $\ell = 0$}
In this case, we are lead to study the asymptotic behaviour, with respect the $W^{1,p}$-weak convergence, as $n \to \infty$ of the rescaled energy $F_n^{0,p}:\U_n^{0,p}\to \mathbb R$ defined in analogy with \eqref{eq:Fn0}, 
\begin{align}
    \label{eq:Fn0p}
    F_n^{0,p}(u_n^a, u_n^b) =  \frac{r_n^2}{h_n} K_n^{a,0}(u_n^a) + K_n^{b,0}(u_n^b), \quad (u_n^a,u_n^b) \in \mathcal{U}^{0,p}_n, \end{align}
where $K_n^{a,0}, K_n^{b,0}$ are defined by \eqref{kna} and \eqref{knb}, taking into account that the loads satisfy

\begin{align}\label{rate 0p}
  H_n^{b,0} \weak H, \hbox{ and } \frac{r^2_n}{h_n}H_n^{a,0}\weak {\bf 0} \;\; \hbox{ in $L^{p'}(\Omega^a,\RI^3)$},
\end{align}

where $H$ and $p'$ are as in \eqref{rateinftyp},
and 
\begin{align}\label{Vn^p_0}
\mathcal{U}^{0,p}_n=\Big\{ &(u^a, u^b) \in \left(c^{a,0}+r_nd^{a,0}x_\alpha 
+W_a^{1,p}
(\Omega^a;\RR^3)\right)\x \left(f^{b,0}+h_n g^{b,0}x_3
+W^{1,p}_b(\Omega^b;\RR^3)\right):
\\ &\qquad 
u^a \text{ and } u^b \text{ satisfy \eqref{eq:bc}}\Big\}, \nonumber
\end{align}
with  the boundary data given by \eqref{dirichlet0} (with $f^{b,0}\in W^{1,p}(\omega;\mathbb R^3)$).

The $W^{1,p}$ counterpart of Proposition \ref{Compactness_0} asserts the following
\begin{Proposition}
    \label{compacntessres0p}
    Consider the Borel function  \(W : \mathbb{R}^{3\x 3}\to 
    \mathbb{R}\) satisfying \eqref{coercip} and \eqref{growthp}.  
    For  $n \in \mathbb N$ let \(r_n\), \(h_n\) be as in \eqref{hrzero}, with $\ell$ in \eqref{ell} coinciding with $0$. Let the spaces
    \(\mathcal{U}_n^{0,p}\) be as in \eqref{Vn^p_0}
    with Dirichlet conditions as in \eqref{dirichlet0}, and the loads $H_n^{a,0},H_n^{b,0}$ satisfy \eqref{rate 0p}.
    Let $F_n^{0,p}$  be the energy functional defined in \eqref{eq:Fn0p}.

 Then, for every  
        $( u_n^a, u_n^b) \in \mathcal{U}^{0,p}_n$  such that \(\sup_{n\in\N} |F_n^{0,p}(  u_n^a,  u_n^b)|
<+\infty\), there exist an increasing sequence of positive integer
        numbers $(n_i)_{i \in \mathbb N}$,
     ${u}^b
        \in f^{b,0}+W^{1,p}_b(\omega;\RI^3)$, 
        depending possibly on
        the selected subsequence $(n_i)_{i \in \mathbb N}$, 
        with $u^b(0_\alpha)=u^a(0)={\bf 0}$ if $p>2$
        such that
      
        \begin{equation*}\left\{
                \begin{array}{l}
                           u^a_{n_i} \to (0_\alpha,x_3) \hbox{ in
                        }W^{1,p}(\Omega^b;\RR^3),\\\\
                    u^b_{n_i}\overset{\ast}{\rightharpoonup}  {u}^b  \hbox{ in
                        }W^{1,p}(\Omega^a;\RR^3),
                        \end{array}\right.\end{equation*}\medskip
        \begin{equation*}
                        \displaystyle{\frac{1}{r_{n_i}}}\nabla_{\alpha}  u^a_{n_i}
                        \to  
                        Id_\alpha
                        \hbox{  in
                        }L^p(\Omega^b;\mathbb R^{3\x2}),
        \end{equation*}\medskip  
        where $Id_\alpha\in \mathbb R^{3\times 2}$ as in \eqref{matrixdef}.
        \end{Proposition}
The proof is omitted, being very similar to the one of Proposition \ref{Compactness_0}, relying in turn on \cite[Proposition 2.1]{GGLM1} and the results contained in \cite[Section 5]{GZNODEA}.

The relaxation now has the form, for $u^b \in f^{b,0}+ W^{1,p}_a(\omega,\RI^3)$

\begin{align*}
G^{0,p}(u^b,(0_\alpha,x_3)):=\inf\left\{\liminf_n\left(\frac{r_n^2}{h_n}K^a_n(  u^a_n)+
K^b_n(  u^b_n)\right): (u_n^a, u_n^b)\in \mathcal U^{0,p}_n,  \right. \\
\left.u_n^a  \to (0_\alpha,x_3) \hbox{ in }W^{1,p}(\Omega^a;\mathbb R^3), 
u_n^b \overset{\ast}{\rightharpoonup} u^b \hbox{ in }W^{1,p}(\Omega^b;\mathbb R^3) \right\}.
\end{align*}

     The following representation result holds
\begin{Proposition}[$\Gamma$-limit $(p,0)$] Under the same assumptions as in Proposition \ref{compacntessres0p}, and with $1<p< 2$, it results
\begin{align*}
G^{0,p}((0_\alpha, x_3), u^b)=\int_{\Omega^b} Q W_0( \nabla_\alpha u^b (x_\alpha))dx + \int_{\Omega^b} H(x)\cdot u^b(x_\alpha) d x_\alpha,
\end{align*}
for every $u^b \in f^{0,b}+ W^{1,p}_a(\omega;\mathbb R^3)$.\end{Proposition}  
\begin{proof}[Proof]
The proof is sketched, (assuming once again for simplicity that the loads are null) relying on arguments entirely similar to the lower bound in Proposition \ref{lbound_0} and the upper bound in Proposition \ref{ub_0_continuous}, taking into account that, with the latter argument we get
$$
G^{0,p}((0_\alpha, x_3), u^b)\leq \int_{\Omega^b} W_0( \nabla_\alpha u^b(x_\alpha) )dx,
$$
for every $u^b \in f^{0,b}+ C^\infty_0(\omega;\mathbb R^3)$, with $u^b(0_\alpha)={\bf 0}$.
The inequality 
$$
G^{0,p}((0_\alpha, x_3), u^b)\leq \int_{\Omega^b} QW_0( \nabla_\alpha u^b (x_\alpha))dx,
$$
for every $u^b \in f^{0,b}+ W^{1,p}_a(\omega;\mathbb R^3),$
is obtained via standard density results as those in \cite[Section 4]{GZNODEA}.
\end{proof}

\section*{Acknowledgements}

Part of this paper has been written during the Erasmus traineeship of GC in the second semester of 2023 at Sapienza-University of Rome and while JM was visiting professor there at Dipartimento di Scienze di Base ed Applicate per l'Ingegneria, whose hospitality both gratefully acknowledge. JM also acknowledges financial support from GNAMPA project, professori visitatori 2022 and FCT/Portugal through CAMGSD, IST-ID, projects UIDB/04459/2020 and UIDP/04459/2020. 
EZ is a member of INdAM-GNAMPA and thanks for its support, through GNAMPA Project 2023 `Prospettive nelle Scienze dei Materiali: modelli variazionali, analisi asintotica e omogeneizzazione'.

\end{document}